\documentclass[12pt,reqno]{amsart}

\usepackage{amssymb, amsmath, amsthm}
\usepackage[bookmarks, colorlinks=true, linkcolor=blue, citecolor=blue, urlcolor=blue]{hyperref}
\usepackage[alphabetic,lite,nobysame]{amsrefs}
\usepackage{verbatim}
\usepackage{amscd}   
\usepackage[all]{xy} 
\usepackage{youngtab} 
\usepackage{young} 
\usepackage{ytableau}
\usepackage{tikz}
\usepackage{ mathrsfs }
\usepackage{cases}
\usepackage{array}
\usepackage{cellspace}
\usepackage{tabu}
\usepackage{calligra,mathrsfs}
\usepackage{bm}

\usepackage[margin=1in]{geometry}

\newcommand{\defi}[1]{{\bf \upshape\sffamily #1}}
\DeclareMathOperator{\ShHom}{\mathscr{H}\text{\kern -3pt {\calligra\large om}}\,}
\newcommand{\arxiv}[1]{\href{http://arxiv.org/abs/#1}{{\tt arXiv:#1}}}
\newcommand{\stacks}[1]{\cite[Tag \href{https://stacks.math.columbia.edu/tag/#1}{#1}]{stacks-project}}

\renewcommand{\a}{\alpha}

\newcommand{\bw}{\bigwedge}

\newcommand{\D}{\operatorname{D}}

\newcommand{\rH}{\mathrm{H}}

\newcommand{\onto}{\twoheadrightarrow}
\newcommand{\oo}{\otimes}

\newcommand{\GL}{\operatorname{GL}}
\newcommand{\Hom}{\operatorname{Hom}}

\newcommand{\rk}{\operatorname{rank}}
\newcommand{\SL}{\operatorname{SL}}

\newcommand{\Sym}{\operatorname{Sym}}
\newcommand{\Tor}{\operatorname{Tor}}

\newcommand{\codim}{\operatorname{codim}}

\renewcommand{\det}{\operatorname{det}}

\renewcommand{\ker}{\operatorname{ker}}

\newcommand{\bb}[1]{\mathbb{#1}}

\renewcommand{\rm}[1]{\textrm{#1}}
\newcommand{\mc}[1]{\mathcal{#1}}

\newcommand{\op}[1]{\operatorname{#1}}

\newcommand{\ul}[1]{\underline{#1}}

\def\kk{{\mathbf k}}

\def\AA{{\mathbf A}}
\def\PP{{\mathbf P}}
\def\lra{\longrightarrow}

\newtheorem{theorem}[equation]{Theorem}
\newtheorem*{theorem*}{Theorem}
\newtheorem*{problem*}{Problem}
\newtheorem{lemma}[equation]{Lemma}

\newtheorem{proposition}[equation]{Proposition}
\newtheorem{corollary}[equation]{Corollary}
\newtheorem*{corollary*}{Corollary}

\theoremstyle{definition}

\newtheorem*{definition*}{Definition}

\theoremstyle{remark}
\newtheorem{remark}[equation]{Remark}
\newtheorem*{remark*}{Remark}

\numberwithin{equation}{section}

\begin{document}

\title{Hermite reciprocity and Schwarzenberger bundles}

\author{Claudiu Raicu}
\address{Department of Mathematics, University of Notre Dame, 255 Hurley, Notre Dame, IN~46556\newline
\indent Institute of Mathematics ``Simion Stoilow'' of the Romanian Academy}
\email{craicu@nd.edu}
\urladdr{\url{https://www3.nd.edu/~craicu/}}
\thanks{CR was supported by NSF DMS-1901886.}

\author{Steven V Sam}
\address{Department of Mathematics, University of California, San Diego}
\email{ssam@ucsd.edu}
\urladdr{\url{http://math.ucsd.edu/~ssam/}}
\thanks{SS was supported by NSF DMS-1849173.}


\date{December 22, 2020}


\begin{abstract} 
  Hermite reciprocity refers to a series of natural isomorphisms involving compositions of symmetric, exterior, and divided powers of the standard $\SL_2$-representation. We survey several equivalent constructions of these isomorphisms, as well as their recent applications to Green's Conjecture on syzygies of canonical curves. The most geometric approach to Hermite reciprocity is based on an idea of Voisin to realize certain multilinear constructions cohomologically by working on a Hilbert scheme of points. We explain how in the case of $\PP^1$ this can be reformulated in terms of cohomological properties of Schwarzenberger bundles. We then proceed to study these bundles from several perspectives:
  \begin{enumerate}
  \item We show that their exterior powers have supernatural cohomology, arising as special cases of a construction of Eisenbud and Schreyer.
  \item We recover basic properties of secant varieties $\Sigma$ of rational normal curves (normality, Cohen--Macaulayness, rational singularities) by considering their desingularizations via Schwarzenberger bundles, and applying the Kempf--Weyman geometric technique.
  \item We show that Hermite reciprocity is equivalent to the self-duality of the unique rank one Ulrich module on the affine cone $\widehat{\Sigma}$ of some secant variety, and we explain how for a Schwarzenberger bundle of rank $k$ and degree $d\geq k$, Hermite reciprocity can be viewed as the unique (up to scaling) non-zero section of $(\Sym^k\mc{E})(-d+k-1)$.
  \end{enumerate}
\end{abstract}

\maketitle

\section{Introduction}\label{sec:intro}

The goal of this article is to uncover a close relationship between Hermite reciprocity and cohomological properties of Schwarzenberger bundles, and to highlight their importance by connecting to a series of recent results in the literature. Specifically, we discuss the applications of Hermite reciprocity to proving Green's conjecture for rational cuspidal curves \cite{AFPRW}, and for canonical ribbons \cite{SR-bigraded}. We also explain how to recover the description of the class group of a Hankel determinantal ring and its property of having rational singularities \cite{CMSV}, by working on the natural desingularizations, via Schwarzenberger bundles, of the secant varieties to rational normal curves.

Classically, Hermite reciprocity is the statement that the composition of two symmetric powers is commutative when applied to a 2-dimensional vector space, that is, there exists a $\GL_2({\bf C})$-equivariant isomorphism \cite[Exercise~11.34]{FH}
\[
\Sym^m(\Sym^n {\bf C}^2) \cong \Sym^n(\Sym^m {\bf C}^2).
\]
The existence of such an isomorphism can be proven combinatorially by computing the characters of both sides. However, the isomorphism is not unique (it can be chosen independently on each isotypic component), and it may not exist if we replace ${\bf C}$ with a field~$\kk$ of arbitrary characteristic, when the representations involved are no longer completely reducible. To get a correct statement for arbitrary fields, one has to replace $\Sym^m$ with the divided power~$\D^m$. We explain in Section~\ref{sec:Hermite} how to construct a commutative diagram of explicit $\SL_2(\kk)$-equivariant isomorphisms (see also \cite[Exercise~11.35]{FH})
\[
\xymatrix{
 & \bw^m(\Sym^{m+n-1}\kk^2) \ar@{-}[dl]_-{\simeq} \ar@{-}[dr]^-{\simeq} & \\
 \D^m(\Sym^n\kk^2)\ar@{-}[rr]^-{\simeq} & & \Sym^n(\D^m \kk^2) 
 }
\]
and we loosely refer to any one of them as \defi{Hermite (reciprocity) isomorphisms}. Notice that we have relaxed the requirement of $\GL_2$-equivariance to $\SL_2$-equivariance: this is only done to avoid twisting by appropriate powers of the \defi{determinant representation} $\det(\kk^2)=\bw^2\kk^2$ of $\GL_2(\kk)$. For instance, to make the diagonal isomorphisms in the above diagram respect the $\GL_2$-action, one would need to tensor the bottom representations with $(\det(\kk^2))^{\oo{m\choose 2}}$.

We give two constructions of Hermite isomorphisms: the first one uses only elementary multilinear algebra, while the second one uses Schwarzenberger bundles and only some elementary facts of algebraic geometry. A third construction that passes through (a truncation of) the ring of symmetric polynomials is explained in \cite[Section~3]{AFPRW}. The fact that all of the constructions agree is a consequence of a strong compatibility between the Hermite isomorphisms as we vary the parameters. For instance, if we vary $n$ then we get
\begin{equation}\label{eq:vary-n}
\bigoplus_{n\geq 0}\bw^m(\Sym^{m+n-1}\kk^2) \simeq \bigoplus_{n\geq 0}\Sym^n(\D^m \kk^2)
\end{equation}
where the right side is manifestly a polynomial ring, the symmetric algebra $\Sym(\D^m \kk^2)$. A natural action of $\D^m\kk^2$ on the left side makes it into a free $\Sym(\D^m \kk^2)$-module of rank one, with generating set $\bw^m(\Sym^{m-1}\kk^2)\simeq \kk$.

The geometric approach to Hermite reciprocity, suggested by Rob Lazarsfeld, is based on a construction considered in \cite[Section~2]{voisin-JEMS}. Specializing it to the case of $\PP^1$, we get an incidence correspondence
\begin{equation}\label{eq:incidence}
\begin{array}{c}
\xymatrix{
Z \ar[r]^-{\pi_1} \ar[d]_-{\pi_2} & \op{Hilb}^{m}(\PP^1) \simeq \PP^{m} \\
\PP^1
}
\end{array}
\end{equation}
where $Z$ consists of pairs $(\Xi,p)$, where $\Xi$ is a subscheme of $\PP^1$ of length $m$ and $p\in\Xi$. Letting
\[ \mc{E} = \pi_{1*}(\pi_2^*(\mc{O}_{\PP^1}(m+n-1))),\]
one has that $\mc{E}$ is a vector bundle of rank $m$ on $\PP^m$, with $\det(\mc{E}) = \bw^m\mc{E} = \mc{O}_{\PP^m}(n)$. It is noted in \cite[(2.11)]{voisin-JEMS} that there exists an isomorphism
\[
  \bw^m \rH^0(\PP^1,\mc{O}_{\PP^1}(m+n-1)) \simeq \rH^0(\PP^m,\det(\mc{E})),
\]
and upon identifying the left side with $\bw^m(\Sym^{m+n-1}\kk^2)$ and the right side with $\Sym^n(\D^m\kk^2)$, one discovers an instance of Hermite reciprocity. We treat this in more detail in Section~\ref{sec:Hermite-Hilbert}, where we note that $\mc{E}=\mc{E}^m_{n-1}$ is a \defi{Schwarzenberger bundle}, with presentation \cite[Proposition~2]{sch}
\[0\lra\Sym^{n-1}\kk^2 (-1) \lra \Sym^{m+n-1}\kk^2 \oo \mc{O}_{\PP^m} \lra \mc{E} \lra 0.\]
In the diagram (\ref{eq:incidence}), in can be shown that $Z\simeq\PP^1\times\PP^{m-1}$, and that the morphism $\pi_1$ is defined by a linear series of type $(1,1)$. As such, $\mc{E}$ arises as a special case of the construction of \cite[Theorem~6.1]{ES-JAMS}, and is therefore a \defi{supernatural vector bundle}. In Section~\ref{sec:Sch} we give a similar realization, which appears to be new, for all of the exterior powers $\bw^i\mc{E}$. We remark here that the natural generalizations of the Schwarzenberger bundles, obtained by replacing $\PP^1$ in \eqref{eq:incidence} with a higher genus curve, have been also notably used in the Ein--Lazarsfeld proof of the gonality conjecture \cite{ein-laz-gonality} and in the recent work of Ein--Niu--Park on secant varieties of nonsingular curves \cite{ENP}.

In the final chapters we discuss two applications of Hermite reciprocity and Schwarzenberger bundles. The first is to the study of secant varieties of rational normal curves, and is developed in Section~\ref{sec:div-group}. These secant varieties have desingularizations given by the total space of Schwarzenberger bundles (see \cite[Section~2]{sch}, \cite[Section~6]{ott-val}, \cite[Section~3]{ENP}).  We review some basic known properties of the secant varieties using what we have developed, such as showing that they are normal Cohen--Macaulay varieties which have rational singularities, and showing that the minimal free resolution of their homogeneous coordinate rings are given by the Eagon--Northcott complex of a Hankel matrix. We also discuss the classification of the rank one maximal Cohen--Macaulay modules on these varieties and show that there is always a distinguished one which is an Ulrich module and is self-dual. Most of the properties of secant varieties that we present are classical, but the results on rational singularities and on the rank one MCMs have been obtained only recently \cite{CMSV}, based on a more algebraic approach. In Section~\ref{sec:self-duality}, we recast Hermite reciprocity in terms of the self-duality of the rank one Ulrich module, and also show that it characterizes the unique global section of an appropriate twist of a symmetric power of $\mc{E}$.

The second application of the ideas behind Hermite reciprocity is to give a purely algebraic proof of Green's conjecture for generic curves of genus $g$, and is discussed in Section~\ref{sec:syzygies}. There are two different approaches \cites{AFPRW,SR-bigraded}, each of which involves studying a mapping cone and showing that certain comparison maps are surjective. Surprisingly, in both cases, by considering all $g$ at once, the Tor groups involved in these comparison maps can be given the structure of finitely generated modules over polynomial rings, in the spirit of (\ref{eq:vary-n}). This hidden structure is most easily observed using Hermite reciprocity, and we explain the calculations leading up to this observation, referring the reader to the original articles for the rest of the technical aspects.

\subsection*{Notation}

Throughout we fix a field ${\bf k}$. Since all of our results are independent of characteristic and compatible with change of rings, one can work over a general commutative ring with the proper rephrasing. We will make use of standard multilinear functors: $\D^m$ denotes the $m$-th divided power functor, which is the subspace of symmetric tensors in the $m$-th tensor power, $\Sym^m$ denotes the $m$-th symmetric power functor, and $\bigwedge^m$ denotes the $m$-th exterior power functor. For any finite dimensional vector space $V$, we have a natural identification $\D^m(V) = \Sym(V^\vee)^\vee$ (see \cite[Section~3.1]{AFPRW} for a quick discussion of the relevant constructions).

\section{Hermite reciprocity}\label{sec:Hermite}

Consider a $2$-dimensional vector space $U$ (which we identify with its dual space $U^{\vee}$ after picking a volume form, i.e., a nonzero element of $\bigwedge^2 U$). We let $\SL=\SL(U) \cong \SL_2(\kk)$ denote the special linear group, i.e., the group of invertible operators on $U$ of determinant~$1$. The careful reader can upgrade all of the statements in this paper to account for the general linear group $\GL(U)$ by inserting appropriate powers of determinant characters, but we avoided this to simplify the notation. In this section, we give two different constructions of the isomorphisms in the following result and then we show that they agree with each other, and with the isomorphisms constructed in \cite[Section~3]{AFPRW}.

\begin{theorem}[Hermite reciprocity]\label{thm:Hermite}
\addtocounter{equation}{-1}
  \begin{subequations}
    We have natural $\SL(U)$-equivariant isomorphisms:
 \begin{equation}\label{eq:Hermite-skew-sym}
  \bw^m(\Sym^n U) \cong \Sym^{n-m+1}(\D^m U),
 \end{equation}
 \begin{equation}\label{eq:Hermite-sym-div}
  \D^m(\Sym^{n-m} U) \cong \Sym^{n-m}(\D^m U).
 \end{equation}
\end{subequations}
\end{theorem}

\subsection{An algebraic construction} \label{sec:hermite-alg}

Given vector spaces $A$ and $B$, we have a canonical map
\begin{equation}\label{eq:Dwegde-wedge}
  \D^m A \otimes \bigwedge^m B \to \bigwedge^m(A \otimes B)
\end{equation}
given by realizing $\bigwedge^m B$ as the skew-symmetric tensors in $B^{\otimes m}$. More precisely, if $\sum_\alpha a_{\alpha_1} \otimes \cdots \otimes a_{\alpha_m}$ is invariant and $\sum_\beta b_{\beta_1} \otimes \cdots \otimes b_{\beta_m}$ is skew-invariant, then $\sum_{\alpha, \beta} (a_{\alpha_1} \otimes b_{\beta_1}) \otimes \cdots \otimes (a_{\alpha_m} \otimes b_{\beta_m})$ is also skew-invariant. Similarly, we have a canonical map
\[
  \D^m A \otimes \D^m B \to \D^m(A \otimes B).
\]

In particular, we can define multiplication maps
\begin{subequations}
  \begin{align}
  \D^m U \otimes \bigwedge^m(\Sym^d U) &\to \bigwedge^m(\Sym^{d+1} U) \label{eqn:hermite-mult}\\
  \D^m U \otimes \D^m(\Sym^d U) &\to \D^m(\Sym^{d+1} U)
\end{align}
\end{subequations}
by using the canonical maps above followed by either the functor $\bigwedge^m$ or $\D^m$ applied to the multiplication map $U \otimes \Sym^d U \to \Sym^{d+1} U$.
If we do it twice, the resulting map is invariant under swapping the copies of $\D^m U$, so $\bigoplus_{d \ge 0} \bigwedge^m (\Sym^d U)$ and $\bigoplus_{d \ge 0} \D^m(\Sym^d U)$ acquire the structure of modules over $\Sym(\D^m U)$.

\begin{proposition} \label{prop:hermite-free}
  If $U$ has basis $\{1,x\}$, then $\bigoplus_{d \ge 0} \bigwedge^m (\Sym^d U)$ is a free $\Sym(\D^m U)$-module of rank one. If $m \ge 1$ then the generator is $x^{m-1} \wedge x^{m-2} \wedge \cdots \wedge 1$ in degree $m-1$, and if $m=0$ then the generator is in degree $0$. Similarly, $\bigoplus_{d \ge 0} \D^m(\Sym^d U)$ is a free $\Sym(\D^m U)$-module of rank one generated in degree~$0$.
\end{proposition}

\begin{proof}
  The $m=0$ case is obvious, so we assume that $m>0$. We claim that the multiplication $\D^m U \otimes \bigwedge^m(\Sym^d U) \to \bigwedge^m(\Sym^{d+1} U)$ is surjective if $d\geq m-1$. First, $\D^m U$ has a basis 
  \[\{x^{(k)} \mid 0 \le k \le m\}\]
  where $x^{(k)}$ is the sum over all $\binom{m}{k}$ ways of tensoring $k$ copies of $x$ and $m-k$ copies of $1$. The multiplication map is then described by
  \[
    x^{(k)} \otimes (x^{d_1} \wedge \cdots \wedge x^{d_m}) \mapsto \sum_{\substack{S \subseteq \{1,\dots,m\}\\ |S|=k}} x^{d'_1} \wedge \cdots \wedge x^{d'_m}
  \]
where $d'_j=d_j$ if $j \notin S$ and $d'_j=d_j+1$ if $j \in S$. Now consider an element of the form $x^{d_1} \wedge \cdots \wedge x^{d_m}$ with $d_1 > \cdots > d_m$. Let $j \ge 0$ be maximal such that $d_j = d+2-j$ (we set $j=0$ if $d_i \ne d+2-i$ for all $i$). We show how $x^{d_1} \wedge \cdots \wedge x^{d_m}$ is in the image of the multiplication map by induction on $j$. If $j=0$, then $d+1 \notin \{d_1,\dots,d_m\}$, so we can multiply $x^{d_1} \wedge \cdots \wedge x^{d_m}$ by $x^{(0)}$. Otherwise, if $j>0$, multiply $x^{d_1-1} \wedge x^{d_2-1} \cdots x^{d_j-1} \wedge x^{d_{j+1}} \cdots x^{d_m}$ by $x^{(j)}$. This is a sum of the term we want together with terms covered by our induction hypothesis, so the claim is proven.

It follows from the prove above that $\bigoplus_{d \ge 0} \bigwedge^m (\Sym^d U)$ is a cyclic $\Sym(\D^m U)$-module, generated by $x^{m-1} \wedge x^{m-2} \wedge \cdots \wedge 1$ in degree $m-1$. Since 
\[\dim \bigwedge^m(\D^{m-1+d} U) ={m+d\choose d} = \dim \Sym^d(\D^m(U))\text{ for all }d\geq 0,\] 
it follows that $\bigoplus_{d \ge 0} \bigwedge^m (\Sym^d U)$ is a free module, as desired.

The symmetric case is similar: we just need to show that the multiplication map $\D^m U \otimes \D^m(\Sym^d U) \to \D^m(\Sym^{d+1} U)$ is surjective. We have a basis of $\D^m(\Sym^d U)$ consisting of the sum of the unique permutations of $x^{d_1} \otimes \cdots \otimes x^{d_m}$ where $d \ge d_1 \ge \cdots \ge d_m \ge 0$, which we denote by $x^{d_1} \cdots x^{d_m}$. The product of $x^{(k)}$ with $x^{d_1} \cdots x^{d_m}$ is almost as before, namely, it is a sum (with coefficients) of $x^{d'_1} \cdots x^{d'_m}$ where $d+1 \ge d'_1 \ge \cdots d'_m \ge 0$ and $(d'_1,\dots,d'_m)$ is obtained by adding $1$ to $k$ of the $d_i$ and sorting. The exact coefficients will not be relevant, rather we will show that $x^{d_1} \cdots x^{d_m}$ is in the image of the multiplication map by induction on how many exponents are equal to $d+1$. If there are none, then this is the product of $x^{(0)}$ and $x^{d_1} \cdots x^{d_m} \in \D^m(\Sym^d U)$. Otherwise, if $d_k=d+1>d_{k+1}$, then consider the product of $x^{(k)}$ with $x^d \cdots x^d x^{d_{k+1}} \cdots x^{d_m} \in \D^m(\Sym^d U)$. It will contain $x^{d+1} \cdots x^{d+1} x^{d_{k+1}} \cdots x^{d_m}$ with coefficient 1, and all other terms will have less than $k$ exponents equal to $d+1$. These latter terms are in the image of the multiplication map by induction, so we are done.
\end{proof}

\begin{proof}[Proof of Theorem~\ref{thm:Hermite}]
  By Proposition~\ref{prop:hermite-free}, the multiplication map
  \[
    \Sym^{n-m+1} (\D^m U) \otimes \bigwedge^m(\Sym^{m-1} U) \to \bigwedge^m(\Sym^n U)
  \]
  is an isomorphism, and $\bigwedge^m(\Sym^{m-1} U)\simeq\kk$ is a trivial $\SL(U)$-representation.
  Similarly, the multiplication map
  \[
    \Sym^{n-m} (\D^m U) \otimes \D^m(\Sym^0 U) \to \D^m(\Sym^{n-m} U)
  \]
  is an isomorphism and $\D^m(\Sym^0 U)\simeq\kk$.
\end{proof}

\subsection{Via Schwarzenberger bundles}\label{sec:Hermite-Hilbert}

We let $\PP^m = \op{Proj}(\Sym(\D^m U))$.
For $d \ge 1$, consider the composition $\Sym^d U \otimes \mc{O}_{\PP^m}(-1)\to \Sym^d U \otimes \Sym^m U \to \Sym^{d+m} U$. This has locally constant rank, so we have a locally free sheaf $\mc{E}_d^m$, the \defi{Schwarzenberger bundle} on $\PP^m$ defined by the short exact sequence
\begin{subequations}
\begin{equation}\label{eq:ses-Edm}
0 \lra \Sym^d U (-1) \lra \Sym^{d+m}U \oo \mc{O}_{\PP^m} \lra \mc{E}_d^m \lra 0.
\end{equation}
It follows from (\ref{eq:ses-Edm}) that
\begin{equation}\label{eq:rank-det-Emd}
\rk(\mc{E}^m_d) = m\quad\mbox{ and }\quad\det(\mc{E}^m_d)= \mc{O}_{\PP^m}(d+1).
\end{equation}
\end{subequations}

\begin{proof}[Proof of Theorem~\ref{thm:Hermite}]
  The $m$-th exterior power of \eqref{eq:ses-Edm} with $d=n-m$ gives a resolution $\mc{F}^n_{\bullet}$ of $\det(\mc{E}^m_{n-m}) \cong \mc{O}_{\PP^m}(n-m+1)$ by locally free sheaves, where
\[
  \mc{F}^n_i = \bw^{m-i}(\Sym^n U) \oo \D^i(\Sym^{n-m}U) \oo \mc{O}_{\PP^m}(-i),\mbox{ for }i=0,\dots,m.
\]

 Since the sheaves $\mc{F}^n_i$ have no cohomology for $i=1,\dots,m$, it follows that
 \[
   \bw^m(\Sym^n U) = \rH^0(\PP^m,\mc{F}^n_0) = \rH^0(\PP^m,\det(\mc{E}^m_{n-m})) = \Sym^{n-m+1}(\D^m U),
 \]
 proving (\ref{eq:Hermite-skew-sym}). To prove (\ref{eq:Hermite-sym-div}), we note that $\det(\mc{E}^m_{n-m}) \oo \mc{O}_{\PP^m}(-1) \cong \mc{O}_{\PP^m}(n-m)$. Since the sheaves $\mc{F}^n_i(-1)$ have no cohomology for $i=0,\dots,m-1$, it follows that
 \[
   \D^m(\Sym^{n-m} U) = \rH^m(\PP^m,\mc{F}^n_m(-1)) = \rH^0(\PP^m,\det(\mc{E}^m_{n-m})(-1)) = \Sym^{n-m}(\D^m U).\qedhere
 \]
\end{proof}

\begin{remark}
\addtocounter{equation}{-1}
  \begin{subequations}
  We have a natural identification between $\PP^m$ and $\op{Hilb}^m(\PP^1)$ (the Hilbert scheme of $m$ points on $\PP^1$), where a point $[f]\in\PP^m$ with $0\neq f\in\Sym^m U$ corresponds to the zero locus of $f\in \rH^0(\PP^1,\mc{O}(m))$. The incidence correspondence
\[
  Z = \{ ([f],[p]) \in \PP^m \times \PP^1 \mid f(p)=0\},
\]
is defined by an exact sequence
\begin{equation}\label{eq:ses-def-Z}
 0 \lra \mc{O}_{\PP^m \times \PP^1}(-1,-m) \lra \mc{O}_{\PP^m \times \PP^1} \lra \mc{O}_Z \lra 0,
\end{equation}
where the defining equation of $Z$ is given by the unique $\SL(U)$-invariant subspace in
\[
  \rH^0(\PP^m \times \PP^1, \mc{O}_{\PP^m \times \PP^1}(1,m)) = \D^m U \oo \Sym^m U.
\]
In coordinates, if $z_0,\dots,z_m$ are the coordinate functions in $\PP^m$, and $x,y$ are those in $\PP^1$, then $Z$ is defined by the equation
\[
  z_0 x^m + z_1 x^{m-1}y + \cdots + z_i x^{m-i}y^i + \cdots +z_m y^m = 0.
\]
We write $\pi_1 \colon Z \lra \PP^m$ and $\pi_2 \colon Z \lra \PP^1$ for the natural projections. Given $n \ge m$, it follows from \eqref{eq:ses-def-Z} that we have
\[
  \mc{E}^m_{n-m} = {\pi_1}_*(\pi_2^* \mc{O}_{\PP^1}(n)) = {\pi_1}_*(\mc{O}_Z(0,n)).
\]
In Section~\ref{sec:Sch} we generalize the isomorphism above to all exterior powers of Schwarzenberger bundles, showing that they can be realized as direct images of line bundles on a product of projective spaces.
\end{subequations}
\end{remark}

\subsection{The isomorphisms agree}

We now prove that the Hermite reciprocity isomorphisms given in Section~\ref{sec:hermite-alg} and Section~\ref{sec:Hermite-Hilbert} agree with each other and also with the one given by \cite[Lemma~3.3]{AFPRW}. We focus on \eqref{eq:Hermite-skew-sym}. The proof for \eqref{eq:Hermite-sym-div} is similar, so we omit it. To do this, we check that there is a commutative diagram
\begin{equation}\label{eq:comm-wedge-sym}
 \begin{gathered}
 \xymatrix{
\bw^m(\Sym^n U) \oo \D^m U \ar[r]\ar[d] &  \bw^m(\Sym^{n+1} U) \ar[d] \\
 \Sym^{n-m+1}(\D^m U)  \oo \D^m U \ar[r] & \Sym^{n-m+2}(\D^m U) \\
 }
 \end{gathered}
\end{equation}
where the vertical maps are isomorphisms induced by the isomorphism in Section~\ref{sec:Hermite-Hilbert}, the bottom map is the natural multiplication, and the top map is \eqref{eqn:hermite-mult}. Since the Hermite reciprocity isomorphisms are characterized by the commutativity of (\ref{eq:comm-wedge-sym}), as seen in the proof of \cite[Lemma~3.3]{AFPRW}, we conclude by verifying the following.

\begin{proposition}\label{prop:comm-diag}
 The diagram (\ref{eq:comm-wedge-sym}) is commutative.
\end{proposition}

\addtocounter{equation}{-1}
\begin{subequations}
  \begin{proof}
    The following square commutes (the vertical maps are the usual multiplication maps)
    \[
      \xymatrix{ 0 \ar[r] & \Sym^{n-m} U \otimes U(-1) \ar[r] \ar[d] & \Sym^n U \otimes U \ar[d] \\
        0 \ar[r] & \Sym^{n-m+1} U (-1) \ar[r] & \Sym^{n+1} U },
    \]
    which induces a map on cokernels
 \begin{equation}\label{eq:EU-to-next-E}
   \mc{E}^m_{n-m} \oo U \lra \mc{E}^{m}_{n+1-m}.
 \end{equation}
 This in turn gives rise to natural maps
  \begin{equation}\label{eq:mult-det-Mb}
  \bw^m\mc{E}^m_{n-m} \oo \D^m U \lra \bw^m(\mc{E}^m_{n-m} \oo U) \lra \bw^m\mc{E}^m_{n+1-m},
  \end{equation}
 which lifts to a map of resolutions $\mc{F}^n_{\bullet} \oo \D^m U \lra \mc{F}^{n+1}_{\bullet}$, as follows. For each $i=0,\dots,m$, the map $\mc{F}^n_{i} \oo \D^m U \lra \mc{F}^{n+1}_{i}$ is given as the composition
 \[
   \bw^{m-i}(\Sym^n U) \oo \D^i(\Sym^{n-m}U) \oo \D^m U \to \left(\bw^{m-i}(\Sym^n U) \oo \D^{m-i}U\right) \oo \left(\D^i(\Sym^{n-m}U) \oo \D^i U\right)
 \]
 \[
   \to \bw^{m-i}(\Sym^n U \oo U) \oo \D^i(\Sym^{n-m}U \oo U) \to \bw^{m-i}(\Sym^{n+1} U) \oo \D^i(\Sym^{n-m+1}U),
 \]
 where all maps are induced by multiplication and comultiplication. The map on global sections
 \[
   \rH^0(\PP^m,\mc{F}^n_{0} \oo \D^m U) \lra \rH^0(\PP^m,\mc{F}^{n+1}_0)
 \]
 is the top map in \eqref{eq:comm-wedge-sym}, while the bottom map in \eqref{eq:comm-wedge-sym} is the map induced from \eqref{eq:mult-det-Mb} by taking global sections, from which the commutativity follows.
\end{proof}
\end{subequations}

\subsection{Compatibility of Hermite isomorphisms}\label{subsec:comp-Her}

We end our discussion of Hermite reciprocity by constructing one last isomorphism and discussing its compatibility with the ones from Theorem~\ref{thm:Hermite}. We define $\alpha$ as the composition
\[
\xymatrix{
\D^m(\Sym^{n-m}U)  \ar[r]^-{\simeq} \ar@/^2pc/[rr]^-\alpha & \D^m(\Sym^{n-m}U) \oo \bw^m(\Sym^{m-1}U) \ar[r] & \bw^m(\Sym^{n-1}U) \\
}
\]
where the first isomorphism follows from the identification $\kk \cong \bw^m(\Sym^{m-1}U)$ given by $1 \mapsto x^{m-1} \wedge \cdots \wedge x \wedge 1$, while the second map is induced by (\ref{eq:Dwegde-wedge}) and the multiplication $\Sym^{n-m}U\oo\Sym^{m-1}U\lra\Sym^{n-1}U$.

\begin{theorem}\label{thm:comp-Her}
 The map $\alpha$ is an isomorphism, and we have a commutative diagram
 \[
 \xymatrix{
 & \D^m(\Sym^{n-m}U) \ar[dl]_{\alpha} \ar[dr]^{\gamma} & \\
 \bw^m(\Sym^{n-1}U) \ar[rr]^-\beta & & \Sym^{n-m}(\D^m U) 
 }
 \]
 where $\beta$ comes from (\ref{eq:Hermite-skew-sym}) and $\gamma$ from (\ref{eq:Hermite-sym-div}).
\end{theorem}

\begin{proof}
  Fix $m$ and consider the direct sum of the terms in the triangle over all $n \ge m$ to get
  \[
 \xymatrix{
 & \bigoplus_{n \ge m} \D^m(\Sym^{n-m}U) \ar[dl]_{\alpha'} \ar[dr]^{\gamma'} & \\
 \bigoplus_{n \ge m} \bw^m(\Sym^{n-1}U) \ar[rr]^-{\beta'} & & \Sym(\D^m U) 
 }
\]
where all 3 terms are free $\Sym(\D^m U)$-modules of rank one, and the maps $\alpha', \beta', \gamma'$ are linear with respect to the $\Sym(\D^m U)$-action. Let $*$ denote the $\Sym(\D^m U)$-action in all cases. For $f \in \Sym(\D^m U)$, we have $\gamma'^{-1}(f) = f * 1$ where $1 \in \kk \cong \D^m(\Sym^0 U)$, and $\beta'^{-1}(f) = x^{m-1} \wedge \cdots \wedge x \wedge 1 \in \bw^m(\Sym^{m-1} U)$. By construction, $\alpha'(1) = x^{m-1} \wedge \cdots \wedge x \wedge 1$, so the triangle commutes.
\end{proof}

\section{Exterior powers of Schwarzenberger bundles}\label{sec:Sch}

The goal of this section is to prove that exterior powers of $\mc{E}^m_d$ arise as special cases of the construction of supernatural vector bundles from \cite[Section~6]{ES-JAMS}.

\begin{theorem}\label{thm:coh-wedge-Sch}
 For $0\leq i\leq m$ consider the multiplication map
 \[
   \mu \colon \PP^i \times \PP^{m-i} \lra \PP^m.
 \]
 We have an isomorphism
 \[ \bw^i\mc{E}_d^m \simeq \mu_*\mc{O}(d+m-i+1,0),\]
 and in particular the bundle $\bw^i\mc{E}_d^m$ has supernatural cohomology, with root sequence
 \[
   -1,-2,\dots,-(m-i),-(m-i+d+2),\dots,-(m+d+1).
 \]
\end{theorem}

If $i=0$ or $i=m$ then there is nothing to prove, so we may assume that $m\geq 2$ and $0<i<m$. Inside the product $\PP=\PP^i \times \PP^m$, consider the locus (see also \cite[Section~2.1.8]{3264})
\[
  Z = \{ (f,g) \mid f \mbox{ divides }g\}.
\]
with the reduced scheme structure. We have an isomorphism
\[
  \phi\colon \PP^i \times \PP^{m-i} \simeq Z \subset \PP,\quad (f,h) \mapsto (f,fh),
\]
and under this isomorphism we have
\[
  \phi^*(\mc{O}_Z(a,b)) = \mc{O}_{\PP^i \times \PP^{m-i}}(a+b,b).
\]

 The divisibility $f|g$ is equivalent to the existence of a form $h\in\Sym^{m-i}U$ and a scalar $c\in\kk$, not both $0$, such that $hf + cg = 0$. It follows that $Z$ can (set-theoretically) be realized as the degeneracy locus (i.e., where the map fails to have full rank) of a map of vector bundles
 \begin{align*}
   \Sym^{m-i}U(-1,0) \oplus \mc{O}_{\PP}(0,-1) &\lra \Sym^m U \oo \mc{O}_{\PP},\\
   (h\oo f, c\oo g) &\mapsto hf+cg.
 \end{align*}
 We write $\alpha \colon \PP\lra\PP^i$ for the first projection, and observe that if we restrict the domain of the morphism above to the first summand then we get from (\ref{eq:ses-Edm}) an injective map
 \[ \Sym^{m-i}U(-1,0) \hookrightarrow \Sym^m U \oo \mc{O}_{\PP},\]
 with cokernel given by $\alpha^*(\mc{E}^i_{m-i})$.

 \begin{lemma} \label{lem:zero-Z}
   $Z$ is the zero scheme of the induced map
 \[
   \mc{O}_{\PP}(0,-1) \lra \alpha^*(\mc{E}^i_{m-i}).
 \]
\end{lemma}

\begin{proof}
Since $\phi$ is a closed immersion, and $\phi^*(\mc{O}_\PP(1,1)) = \mc{O}_{\PP^i \times \PP^{m-i}}(2,1)$, we can compute the degree of $Z$ with respect to $\mc{O}_\PP(1,1)$ as the $m$-fold self-intersection of $\mc{O}_{\PP^i \times \PP^{m-i}}(2,1)$. This is the coefficient of $s^it^{m-i}$ in $(2s+t)^m$, which is $2^i \binom{m}{i}$. 
  
  Since $\codim Z = i$, the cohomology class of the zero locus of this section is the top Chern class of $\alpha^*(\mc{E}^i_{m-i})(0,1)$. Writing the Chow ring of $\PP$ as $\mathbb{Z}[s,t]/(s^{i+1},t^{m+1})$ (using \cite[Theorem~2.10]{3264}), the top Chern class of $\alpha^*(\mc{E}^i_{m-i})(0,1)$ is by \cite[Proposition~5.17]{3264}
 \[\sum_{j=0}^i c_j(\mc{E}_{m-i}^i,s) t^{i-j}\]
where $c_j(\mc{E}_{m-i}^i,s)$ denotes the $j$th Chern class of $\alpha^*(\mc{E}_{m-i}^i)$. Using \eqref{eq:ses-Edm} and \cite[Theorem 5.3(c)]{3264}, the Chern polynomial of $\alpha^*(\mc{E}_{m-i}^i)$ is (see also \cite[Section~9.3.3]{3264})
\[(1-s)^{m-i+1} = \sum_{j=0}^i \binom{m-i+j}{j} s^j.\]
Since $Z$ has dimension $m$, the degree of the zero locus with respect to $\mc{O}_\PP(1,1)$ is the coefficient of $s^it^m$ in $(1-s)^{m-i+1} (s+t)^m$, which is
   \[
     \sum_{j=0}^i \binom{m-i+j}{j} \binom{m}{i-j} = \sum_{j=0}^i \binom{m}{i} \binom{i}{j} = 2^i \binom{m}{i}.
   \]
   This agrees with the degree of $Z$, and hence we conclude that $Z$ is scheme-theoretically the zero locus of the claimed map of vector bundles.
 \end{proof}

 \begin{proof}[Proof of Theorem~\ref{thm:coh-wedge-Sch}]
We will prove the result by induction on $m$. From Lemma~\ref{lem:zero-Z}, we obtain an exact Koszul resolution (using (\ref{eq:rank-det-Emd}))
\[
  0 \lra \mc{O}_{\PP}(0,-i) \lra \alpha^*(\mc{E}^i_{m-i})(0,-i+1) \lra \cdots \lra \alpha^*(\det(\mc{E}^i_{m-i})) \lra \mc{O}_Z(m-i+1,0)\lra 0.
\]

If we let $\beta\colon\PP\lra\PP^m$ denote the second projection, then $\mu = \beta\circ\phi$, and in particular
\[ \mc{F} := \mu_*(\mc{O}_{\PP^i \times \PP^{m-i}}(d+m-i+1,0)) = \beta_*(\mc{O}_Z(d+m-i+1,0))\]
is resolved by the push-forward along $\beta$ of the earlier Koszul complex twisted by $\mc{O}_{\PP}(d,0)$:
\begin{align} \label{eqn:koszulE}
  0 \lra \mc{O}_{\PP}(d,-i) \lra \alpha^*(\mc{E}^i_{m-i})(d,-i+1) \lra \cdots \lra \alpha^*(\det(\mc{E}^i_{m-i}))(d,0).
\end{align}
By induction, we know that $(\bw^j\mc{E}^i_{m-i})(d)$ has no higher cohomology, and
\[ 
\begin{aligned}
\rH^0(\PP^i,\bw^j(\mc{E}^i_{m-i})(d)) &= \Sym^{m-j+1+d}(\D^j U) \oo \Sym^d(\D^{i-j}U) \\
&\cong \bw^j(\Sym^{m+d}U) \oo \D^{i-j}(\Sym^d U),
\end{aligned}
\]
where the last isomorphism follows from Hermite reciprocity. Since $\beta$ is  a finite map, we get that $\beta_*$ is exact, and hence the sheaf $\mc{F}$ is resolved by a complex
\[
  \cdots \lra \bw^j(\Sym^{m+d}U) \oo \D^{i-j}(\Sym^d U) (-i+j) \lra
  \cdots \lra \bw^i(\Sym^{m+d}U) \oo \mc{O}_{\PP^m}.
\]

We claim that the rightmost differential in the above complex can be identified with the rightmost differential in the $i$-th exterior power of the $2$-term resolution 
\[
  \Sym^d U(-1)\lra \Sym^{m+d}U\oo \mc{O}_{\PP^m}
\]
of $\mc{E}^m_d$. Once shown, this implies that $\mc{F} \simeq \bw^i \mc{E}^m_d$.

To prove the claim, in the exact sequence \eqref{eqn:koszulE}, replace each term $\alpha^*(\bw^j \mc{E}^i_{m-i})$ by its resolution $\alpha^*(\bw^j(\Sym^{m-i} U (-1) \to \Sym^m U))$. Then we get a double complex mapping to the complex in question, and we take sections of the rightmost two terms to get:
  \[
    \xymatrix{     0 & 0 \\
      \bw^{i-1}(\Sym^{m+d} U) \otimes \Sym^d U(-1) \ar[r] \ar[u] & \bw^i(\Sym^{m+d}U) \oo \mc{O}_{\PP^m}\ar[u] \\
      \bw^{i-1}(\Sym^{m} U) \otimes \Sym^d (\D^i U)(-1) \ar[r] \ar[u] & \bw^i(\Sym^m U) \otimes \Sym^d(\D^i U) \oo \mc{O}_{\PP^m} \ar[u] \\
      \vdots \ar[u] & \vdots \ar[u]
    }
  \]
  The vertical maps from the second row to the top row are surjective and the second row comes from the $i$th exterior power of the 2-term complex $\mc{O}(-1) \to \Sym^m U$ tensored with $\Sym^d(\D^i U)$. This implies that the differentials in the first row are determined by the second row and the vertical maps, so it suffices to show that the claimed differential for the first row gives a commutative square.

  For $j=i,i-1$, the vertical map
  \[
    \bw^j(\Sym^{m+d} U) \otimes \Sym^d(\D^i U)(-i+j) \to \bw^j(\Sym^{m+d} U)\otimes \D^{i-j}(\Sym^d U)(-i+j)
  \]
  factors as
  \begin{align*}
    \bw^j(\Sym^{m+d} U) \otimes \Sym^d(\D^i U)(-i+j) &\to \bw^j(\Sym^{m+d} U) \otimes \Sym^d(\D^j U) \otimes \Sym^d(\D^{i-j} U) (-i+j)\\
                                                     &\to \bw^j(\Sym^{m+d} U)\otimes \D^{i-j}(\Sym^d U)(-i+j)
  \end{align*}
  where in the second map we use the action of $\Sym(\D^j U)$ on $\bigoplus_{n\geq 0} \bw^j (\Sym^n U)$ from the previous section on the first two factors (which we denote by $*$). The second factor is the identity map in both cases.

  It suffices to consider the case $d=1$ due to the associativity of $*$. The square becomes
  \[
        \xymatrix{  \bw^{i-1}(\Sym^{m+1} U) \otimes U(-1) \ar[r] & \bw^i(\Sym^{m+1}U) \oo \mc{O}_{\PP^m}\\
 \bw^{i-1}(\Sym^{m} U) \otimes \D^i U(-1) \ar[r] \ar[u] & \bw^i(\Sym^m U) \otimes \D^i U \oo \mc{O}_{\PP^m} \ar[u] 
    }.
  \]
  Pick $\omega \otimes x^{(j)} \otimes f \in \bw^{i-1}(\Sym^m U) \otimes \D^i U(-1)$. The bottom path is
  \begin{align*}
    \omega \otimes x^{(j)} \otimes f \mapsto \omega \wedge f \otimes x^{(j)} \mapsto x^{(j)} * (\omega \wedge f),
  \end{align*}
  while the top path is
  \begin{align*}
    \omega \otimes x^{(j)} \otimes f &\mapsto x^{(j-1)} * \omega \otimes x \otimes f + x^{(j)} * \omega \otimes 1\otimes f\\
    &\mapsto (x^{(j-1)} * \omega) \wedge x f + (x^{(j)} * \omega) \wedge  f
  \end{align*}
  where by convention, $x^{(-1)}=0$. The two final quantities agree, which proves the claim.
\end{proof}

\section{Secant varieties of rational normal curves}\label{sec:div-group}

In this section we give an $\SL$-equivariant construction of the rank one maximal Cohen--Macaulay modules over a Hankel determinantal ring $B$, and recover the description of the divisor class group of $B$ from \cite[Section~3]{CMSV}, as well as the property of $B$ having rational singularities. We formulate our results and arguments geometrically, using the usual identification of $\op{Spec}(B)$ with the affine cone $\widehat{\Sigma}$ over a secant variety of a rational normal curve. In the process we recover well-known properties of $\widehat{\Sigma}$, such as normality and the Cohen--Macaulay property, along with the explicit description of its equations and syzygy modules. The key ingredients that we employ are the desingularization of $\widehat{\Sigma}$ via Schwarzenberger bundles, as explained in \cite[Section~6]{ott-val}, and the Kempf--Weyman technique for constructing syzygies, as explained in \cite[Chapter~5]{weyman}. We will assume that $\kk$ is algebraically closed in order to make valid set-theoretic arguments involving the $\kk$-points of our varieties, but the careful reader may wish to rephrase the justifications in order to remove this hypothesis.

Before going into more details, we establish some notation used throughout the section. If $\mc{F}$ is a coherent locally free sheaf on a variety $X$, we consider the sheaf of (graded) algebras
\[ \Sym_{\mc{O}_X}(\mc{F}) = \mc{O}_X \oplus \mc{F} \oplus \Sym^2(\mc{F}) \oplus \cdots\]
and write $\bb{P}_X(\mc{F})$ for $\ul{\op{Proj}}_X(\Sym_{\mc{O}_X}(\mc{F}))$. Similarly, we write $\bb{A}_X(\mc{F})$ for $\ul{\op{Spec}}_X(\Sym_{\mc{O}_X}(\mc{F}))$. When $X$ is understood from the context, we simply write $\Sym(\mc{F})$, $\bb{P}(\mc{F})$, and $\bb{A}(\mc{F})$.

We let $\PP^n=\bb{P}(\Sym^n U)$, and note that its $\kk$-points $[f]\in \PP^n$ are represented by non-zero elements $f\in \D^n U$ up to scaling. Every $u\in U$ gives rise to a symmetric tensor
\[ u^{(n)} = u \oo u \oo \cdots \oo u \in \D^n U,\]
the \defi{$n$-th divided power} of $u$. We get an $\SL$-equivariant map
\[
  \PP^1 \lra \PP^n,\quad [u] \lra [u^{(n)}],
\]
called the degree $n$ \defi{Veronese embedding} of $\PP^1$. We denote its image by $\Gamma$, which is a \defi{rational normal curve} of degree $n$, and write $\Sigma_k$ for the $k$-secant variety of $\Gamma$. Recall that this is the Zariski closure of the union of all linear spaces $\rm{Span}(x_1,\dots,x_k)$, ranging over all choices of points $x_1,\dots,x_k \in \Gamma$. In particular, we have $\Sigma_1 = \Gamma$. We consider the affine space $\AA^{n+1}= \bb{A}(\Sym^n U)$, and write $\widehat{\Sigma}_k\subset\AA^{n+1}$ for the affine cone over~$\Sigma_k$. We let $B$ denote the coordinate ring of $\widehat{\Sigma}_k$ (or the homogeneous coordinate ring of $\Sigma_k$), which is called a Hankel determinantal ring in \cite{CMSV}. The following theorem summarizes some of the basic properties of~$\widehat{\Sigma}_k$:

\begin{theorem}\label{thm:Hankel}
 The variety $\widehat{\Sigma}_k$ is normal, Cohen--Macaulay, and has rational singularities. Its divisor class group is isomorphic to $\bb{Z}/(n-2k+2)\bb{Z}$.
\end{theorem}

The conclusions of Theorem~\ref{thm:Hankel} are not new. Our goal to provide a unified proof of this theorem based on properties of Schwarzenberger bundles and the Kempf--Weyman geometric technique. For a more in-depth study of other aspects of the theory, the reader can consult \cites{GP,eis-lin-sec,watanabe,conca,CMSV} or \cite[Section~10.4]{3264}.

We consider the Schwarzenberger bundle $\mc{E} = \mc{E}^k_{n-k}$ on $\PP^k = \bb{P}(\D^k U)$, along with the diagram of spaces and maps
\[
\xymatrix{
& Y = \bb{A}_{\PP^k}(\mc{E}) \ar[d]^-{\iota} \ar[dl]_-{\pi} \ar[dr]^-\psi \\
\AA^{n+1} &  \AA^{n+1} \times \PP^k \ar[l]^-q \ar[r]_-p & \PP^k
}
\] 
where $\iota$ is the closed immersion induced by the surjection $\Sym^n U \oo \mc{O}_{\PP^k} \onto \mc{E}$ in (\ref{eq:ses-Edm}), the maps $p,q$ are the projections to the two factors, and $\pi = q\circ\iota$. The map $\psi=p\circ\iota$ is the structure map of the geometric vector bundle $Y$ over $\PP^k$, whose $\kk$-points correspond to pairs $(f,[g])$, where $f$ belongs to the dual of the fiber of $\mc{E}$ at $[g]\in\PP^k$. To make this more explicit, we use the perfect pairing
\[
  \langle-,-\rangle\colon\D^d U \times \Sym^d U \lra \kk,
\]
which exists for each $d\geq 0$ and is $\SL$-equivariant. It will be important to note that if we think of $P \in \Sym^d U$ as a homogeneous polynomial of degree $d$ on $U^{\vee}=U$, then for each $u\in U$, the evaluation of $P$ at $u$ is computed by
\begin{subequations}
\begin{equation}\label{eq:Pu}
 P(u) = \langle u^{(d)},P\rangle.
\end{equation}
We can construct more generally a \defi{contraction map}
\[
  \langle-,-\rangle\colon\D^d U \times \Sym^r U \lra \D^{d-r}U\text{ for }d\geq r\geq 0,
\]
induced by the comultiplication $\D^d U \to \D^{d-r}U \oo \D^r U$ and the pairing $\D^r U \times \Sym^r U \lra \kk$. Suppose now that $[g]\in\PP^k$, where $0\neq g\in\Sym^k U$. If we restrict (\ref{eq:ses-Edm}) to the fiber at $[g]$ and dualize, we can identify the fiber of $\mc{E}^{\vee}$ at $[g]$ with the kernel of the contraction
\begin{equation}\label{eq:contr-g}
 \langle -,g \rangle\colon \D^n U \lra \D^{n-k}U.
\end{equation}
\end{subequations}
This yields the explicit description of the $\kk$-points in $Y$ as
\begin{equation}\label{eq:explicit-Y}
 Y = \{(f,[g]): 0\neq g\in\Sym^k U,\ f\in\ker\langle -,g \rangle\colon \D^n U \lra \D^{n-k}U\}.
\end{equation}
The connection between Schwarzenberger bundles and secant varieties is given as follows.

\begin{lemma}\label{lem:desingularization}
 The image of $\pi$ is $\widehat{\Sigma}_k\subseteq\AA^{n+1}$.
\end{lemma}

\begin{proof} Consider a general point $(f,[g])\in Y$, where $g\in\Sym^k U$ is a homogeneous polynomial with distinct roots $u_1,\dots,u_k\in U$. It follows from (\ref{eq:Pu}) that $u_1^{(n)},\dots,u_k^{(n)}$ belong to the kernel of the map (\ref{eq:contr-g}), and since they are linearly independent, they must generate the fiber of $\mc{E}^{\vee}$ at $[g]$. We conclude that $f\in\op{Span}\{u_1^{(n)},\dots,u_k^{(n)}\}$, hence $f\in\widehat{\Sigma}_k$. Conversely, since every general point $f\in\widehat{\Sigma}_k$ belongs to $\op{Span}\{u_1^{(n)},\dots,u_k^{(n)}\}$ for some distinct $u_1,\dots,u_k\in U$, it follows by considering $g\in\Sym^k U$ with roots $u_1,\dots,u_k$, that $f\in\op{Im}(\pi)$. This shows that $\op{Im}(\pi)$ is dense in $\widehat{\Sigma}_k$, but since $\pi$ is a projective morphism, it follows that $\op{Im}(\pi)=\widehat{\Sigma}_k$.
\end{proof}

We will see shortly that $\pi$ is in fact birational, and therefore it provides a resolution of singularities of $\widehat{\Sigma}_k$. We make the usual identification between quasi-coherent sheaves on affine space and their global sections, and let
\[ S = \mc{O}_{\AA^{n+1}},\quad B = \mc{O}_{\widehat{\Sigma}_k},\quad \tilde{B} = \pi_*\mc{O}_Y.\]

\begin{proposition}\label{prop:syzygies-B}
 We have that $\tilde{B}=B$ has an $\SL$-equivariant minimal graded free resolution $F_{\bullet}$ over $S$, whose terms are given by
 \[ F_0 = S,\quad F_i =  \D^{i-1}(\Sym^k U) \oo \bw^{i+k}(\Sym^{n-k}U) \oo S(-i-k)\text{ for }i=1,\dots,n-2k+1.\]
\end{proposition}

\begin{proof} Since the natural map $S\lra\tilde{B}$ factors through $B$, in order to prove that $\tilde{B}=B$, it suffices to check that $S$ surjects onto $\tilde{B}$. We do so by applying \cite[Theorem~5.1.2]{weyman} with $V=\PP^k$, $\mc{V}=\mc{O}_{\PP^k}$, $X=\AA^{n+1}$, $\eta=\mc{E}$, and $\xi=\Sym^{n-k}U(-1)$. We get a complex $F_{\bullet}$ of free $S$-modules, with 
\[
\begin{aligned}
F_i &= \bigoplus_{j\geq 0} \rH^j\left(\PP^k,\bw^{i+j}(\Sym^{n-k}U(-1))\right) \oo S(-i-j) \\
&= \bigoplus_{j\geq 0}\rH^j\left(\PP^k,\mc{O}_{\PP^k}(-i-j)\right) \oo \bw^{i+j}(\Sym^{n-k}U) \oo S(-i-j),
\end{aligned}
\]
whose homology groups vanish in positive degrees, and satisfy
\[
  \rH_{-i}(F_{\bullet}) = \rH^i(Y,\mc{O}_Y) = \bigoplus_{d\geq 0} \rH^i(\PP^k,\Sym^d\mc{E}), \text{ for }i\geq 0.
\]

To identify the terms in the complex $F_{\bullet}$, we note that a line bundle $\mc{O}_{\PP^k}(d)$ has no intermediate cohomology and
\begin{align*}
\rH^0(\PP^k,\mc{O}_{\PP^k}(d)) &= \begin{cases}
\Sym^d(\D^k U) & \text{if }d\geq 0, \\
0 & \text{otherwise},
\end{cases}\\
\rH^k(\PP^k,\mc{O}_{\PP^k}(d)) &= \begin{cases}
\D^{-d-k-1}(\Sym^k U) & \text{if }d\leq -k-1, \\
0 & \text{otherwise}.
\end{cases}
\end{align*}
It follows that $F_0=S$, and the only other non-zero terms are 
\[
  F_i = \D^{i-1}(\Sym^k U) \oo \bw^{i+k}(\Sym^{n-k}U) \oo S(-i-k)\text{ for }i=1,\dots,n-2k+1.
\]
We conclude that $\rH_i(F_\bullet)=0$ for $i<0$, and that $F_{\bullet}$ gives a minimal free resolution of $\tilde{B}$. Since $F_0=S$, the natural map $S\lra \tilde{B}$ is surjective, hence $\tilde{B}=B$, as desired.
\end{proof}

To make the results of Proposition~\ref{prop:syzygies-B} even more explicit, consider for a moment the general situation of a $\kk$-linear map
\[
  \varphi \colon V_1 \lra V_0 \oo W,
\]
where $W,V_0,V_1$ are finite dimensional $\kk$-vector spaces of dimensions $k+1,n+1$, and $m+1$ respectively. After choosing bases for $W,V_0,V_1$, we can represent $\varphi$ either as an $(m+1)\times(n+1)$ matrix $A$ of linear forms in $\Sym(W)\simeq\kk[y_0,\dots,y_k]$, or as an $(m+1)\times(k+1)$ matrix $A'$ of linear forms in $\Sym(V_0)=\kk[z_0,\dots,z_n]$. Such matrices occur frequently, for instance in the study of Rees algebras, when one is the presentation matrix of an ideal, while the other is the \defi{Jacobian dual} \cites{vasc,SUV}. Of interest to us is the encoding of $\varphi$ as a morphism of sheaves on $\bb{P}W\simeq\PP^k$
\begin{equation}\label{eq:linear-map}
 V_1 \oo \mc{O}_{\bb{P}W}(-1) \lra V_0 \oo \mc{O}_{\bb{P}W}.
\end{equation}
Taking the $(k+1)$-st symmetric power yields a Koszul complex $\mc{K}_{\bullet}$
\[
  0\lra \bw^{k+1}V_1 \oo \mc{O}_{\bb{P}W}(-k-1) \lra \cdots \lra V_1\oo\Sym^k V_0 \oo \mc{O}_{\bb{P}W}(-1) \lra \Sym^{k+1}V_0 \oo \mc{O}_{\bb{P}W}\lra 0.
\]
The intermediate sheaves in the above complex have no cohomology, and the hypercohomology spectral sequence involves precisely one interesting map
\begin{equation}\label{eq:hyper-map}
\xymatrix{
\rH^k\left(\bb{P}W, \bw^{k+1}V_1 \oo \mc{O}_{\bb{P}W}(-k-1)\right) \ar[r] \ar@{=}[d] & \rH^0(\bb{P}W,\Sym^{k+1}V_0 \oo \mc{O}_{\bb{P}W}) \ar@{=}[d] \\
\bw^{k+1}V_1 \simeq \bw^{k+1}V_1 \oo \bw^{k+1}W^{\vee} & \Sym^{k+1}V_0
}
\end{equation}
If we think of the basis of $\bw^{k+1}V_1$ as being indexed by collections of $(k+1)$ rows of the matrix $A'$, then this map associates to every such collection the corresponding (maximal) $(k+1)\times(k+1)$ minor of $A'$. 

We specialize this discussion to the case when $V_1=\Sym^{n-k}U$ (where $m=n-k+1$), $V_0=\Sym^n U$, $W=\D^k U$, and the map $\varphi$ is induced by the dual of the contraction map discussed earlier. If we choose the standard monomial bases on the three vector spaces, and denote them $x_{\bullet}$ on $V_1$, $z_{\bullet}$ on $V_0$, $y_{\bullet}$ on $W$, then we have
\begin{equation}\label{eq:pres-E}
 \varphi(x_i) = \sum_{j=0}^k z_{i+j} \oo y_j\text{ for }i=0,\dots,n-k.
\end{equation}
The matrix $A'$ then takes the form
\[
  A'=
\begin{bmatrix}
z_0 & z_1 & \cdots & z_k \\
z_1 & z_2 & \cdots & z_{k+1} \\
\vdots & \vdots & \ddots & \vdots \\
z_{n-k} & z_{n-k+1} & \cdots & z_n
\end{bmatrix}
\]
which is called a \defi{Hankel} (or \defi{catalecticant}) matrix. The map (\ref{eq:linear-map}) is the presentation of $\mc{E}$ from (\ref{eq:ses-Edm}), and therefore $\mc{K}_{\bullet}$ is a resolution of $\Sym^{k+1}\mc{E}$. The map (\ref{eq:hyper-map}) is the degree $(k+1)$ component of the differential $d_1 \colon F_1\lra F_0$ (whose cokernel is $B_{k+1}=\rH^0(\PP^k,\Sym^{k+1}\mc{E})$). Since $F_1$ is generated by its degree $(k+1)$ component, it follows that the image of $d_1$ (which is the defining ideal of $\widehat{\Sigma}_k$) is the ideal $I$ generated by the $(k+1)\times(k+1)$ minors of the Hankel matrix $A'$. The complex $F_{\bullet}$ is the \defi{Eagon--Northcott complex} associated with $A'$.

\begin{corollary}\label{cor:dim-degree-Sk} The variety $\Sigma_k$ is arithmetically Cohen--Macaulay. Its dimension and degree are computed by:
  \[
    \dim(\Sigma_k) = 2k-1,\quad \deg(\Sigma_k) = {n-k+1\choose k}\text{ for }1\leq k\leq\frac{n+1}{2}.
  \]
\end{corollary}

\begin{proof}
  \addtocounter{equation}{-1}
  \begin{subequations}
 We will prove shortly that the Hilbert series of $B$ can be expressed in lowest terms as
\begin{equation}\label{eq:HS-B}
 \op{HS}_B(t) = \frac{\sum_{i=0}^k {n-2k+i\choose i} \cdot t^i}{(1-t)^{2k}}.
\end{equation}
This implies that $\dim(\Sigma_k)+1 = \dim(B) = 2k$, and setting $t=1$ in the numerator we obtain
\[ \deg(\Sigma_k) = \sum_{i=0}^k {n-2k+i\choose i} = {n-k+1\choose k}\]
(the sum is the number of monomials of degree $\le k$ in $(n+1-2k)$ variables, which by homogenization, is the number of monomials of degree exactly $k$ in $(n+2-2k)$ variables, which is the right side). Since $B$ has codimension $(n+1-2k)$, equal to the projective dimension as computed by the resolution in Proposition~\ref{prop:syzygies-B}, it follows that $B$ is a Cohen--Macaulay module, hence $\Sigma_k$ is arithmetically Cohen--Macaulay.

To prove (\ref{eq:HS-B}), we use the minimal free resolution of $B$ from Proposition~\ref{prop:syzygies-B} to obtain
\[ \op{HS}_B(t) = \frac{1 + \sum_{i=1}^{n-2k+1} (-1)^i \binom{i-1+k}{k} \binom{n-k+1}{i+k} \cdot t^{i+k}}{(1-t)^{n+1}}.\]
If we write $F(t)$ for the numerator, then (\ref{eq:HS-B}) is equivalent to $F(t)=G(t)$, where
\[ G(t) = (1-t)^{n+1-2k}\cdot\sum_{i=0}^k {n-2k+i\choose i} \cdot t^i.\]
Since $F(t)$ and $G(t)$ have constant term $1$, to show they coincide it suffices to check that $F'(t)=G'(t)$. We have
\[ 
\begin{aligned}
F'(t) &= \sum_{i=1}^{n-2k+1}(-1)^i {i-1+k\choose k}{n-k+1\choose i+k}(i+k)\cdot t^{i+k-1} \\
&=\sum_{i=1}^{n-2k+1}(-1)^i {n-k\choose k}{n-2k\choose i-1}(n-k+1)\cdot t^{i+k-1} \\
&\overset{j=i-1}{=} -(n-k+1){n-k\choose k}\cdot t^k\cdot\sum_{j=0}^{n-2k}(-1)^j{n-2k\choose j}\cdot t^j \\
& = -(n-k+1){n-k\choose k}\cdot t^k\cdot (1-t)^{n-2k}.
\end{aligned}
\]
Using the product rule for $G'(t)$ and dividing by $(1-t)^{n-2k}$, we obtain
\[
\frac{G'(t)}{(1-t)^{n-2k}} = -(n-2k+1)\cdot\left(\sum_{i=0}^k {n-2k+i\choose i} \cdot t^i\right) + (1-t)\cdot\left(\sum_{i=0}^k {n-2k+i\choose i} i \cdot t^{i-1}\right).
\]
In the above sum, the coefficient of $t^i$ vanishes for $0\leq i\leq k-1$ due to the identity
\[ -(n-2k+1){n-2k+i\choose i} + {n-2k+i+1\choose i+1}(i+1) - {n-2k+i\choose i}i = 0,\]
while the coefficient of $t^k$ is
\[-(n-2k+1){n-k\choose k} - {n-k\choose k} k = -(n-k+1){n-k\choose k}.\]
This shows that $F'(t)=G'(t)$, concluding the proof.
\end{subequations}
\end{proof}

\begin{corollary}\label{cor:pi-birational}
 The morphism $\pi$ is birational onto $\widehat{\Sigma}_k$. The variety $\widehat{\Sigma}_k$ is normal with rational singularities.
\end{corollary}

\begin{proof}
 Since $\dim(Y) = \dim(\widehat{\Sigma}_k)$, it follows that $\pi\colon Y\lra\widehat{\Sigma}_k$ is generically finite. Combining this with the equality $\pi_*\mc{O}_Y = \mc{O}_{\widehat{\Sigma}_k}$ established in Proposition~\ref{prop:syzygies-B}, we see that $\pi$ is birational \stacks{03H2}, hence it provides a resolution of singularities of $\widehat{\Sigma}_k$. The fact that $\widehat{\Sigma}_k$ is normal with rational singularities is now a consequence of \cite[Theorem~5.1.3(c)]{weyman}, and of the description of the sheaf cohomology groups of $\mc{O}_Y$ explained in the proof of Proposition~\ref{prop:syzygies-B}.
\end{proof}

Returning to the explicit presentation (\ref{eq:pres-E}) of the bundle $\mc{E}$, and using the affine coordinates $z_0,\dots,z_n$ on $\AA^{n+1}$, and the projective coordinates $y_0,\dots,y_k$ on $\PP^k$ as before, it follows that $Y$ is defined as a subvariety in $\AA^{n+1}\times\PP^k$ by the condition
\[ \begin{bmatrix}
z_0 & z_1 & \cdots & z_k \\
z_1 & z_2 & \cdots & z_{k+1} \\
\vdots & \vdots & \ddots & \vdots \\
z_{n-k} & z_{n-k+1} & \cdots & z_n
\end{bmatrix} \cdot 
\begin{bmatrix}
y_0 \\ y_1 \\ \vdots \\ y_k
\end{bmatrix}
=
\vec{0}.
\]
Representing each element of $\AA^{n+1}$ as a Hankel matrix $M$, and each element of $\PP^k$ as a non-zero column vector $\vec{v}$ up to scaling, we get
\[
  Y = \{ (M,[\vec{v}]) : M\cdot\vec{v}=\vec{0}\},\text{ and } \pi \colon Y\lra\AA^{n+1}\text{ is given by }\pi(M,[\vec{v}])=M.
\]
Since the generic element $M\in\widehat{\Sigma}_k$ has rank $k$, we have that $\ker(M)$ is one-dimensional, hence the non-zero vector $\vec{v}$ in $\ker(M)$ is uniquely defined up to scaling. This gives a more concrete interpretation of the fact that $\pi$ is birational onto $\widehat{\Sigma}_k$. It also shows that the locus where $\pi$ fails to be an isomorphism is identified with the set of Hankel matrices of rank $\leq(k-1)$, that is, with $\widehat{\Sigma}_{k-1}$. We are now ready to prove the final conclusion of Theorem~\ref{thm:Hankel}.

\begin{proposition}\label{prop:class-group}
 The class group of $\widehat{\Sigma}_k$ is isomorphic to $\bb{Z}/(n-2k+2)\bb{Z}$.
\end{proposition}

\begin{proof}
  \addtocounter{equation}{-1}
  \begin{subequations}
    We let $U=\widehat{\Sigma}_k \setminus \widehat{\Sigma}_{k-1}$, define
\[ Z= \pi^{-1}(\widehat{\Sigma}_{k-1}) \subset Y,\]
and note that as remarked earlier that $\pi$ establishes an isomorphism between $Y\setminus Z$ and $U$. We will show that $Z$ is a divisor on $Y$, defined by a section of $\psi^*(\mc{O}_{\PP^k}(-n+2k-2))$ (see also \cite[Proposition~6.14]{ott-val}). Since $\widehat{\Sigma}_{k-1}$ has codimension two inside $\widehat{\Sigma}_{k}$, it follows that
\[ \op{Cl}(\widehat{\Sigma}_{k}) = \op{Cl}(U)=\op{Cl}(Y\setminus Z).\]
Moreover, we have an exact sequence
\[ \bb{Z} \lra \op{Cl}(Y) \lra \op{Cl}(Y\setminus Z) \lra 0,\]
where the first map sends $1$ to $[Z]$ \stacks{02RX}. Since $Y$ is a vector bundle over $\PP^k$, it follows that $\op{Cl}(Y)=\bb{Z}$, generated by the pullback of a hyperplane class along $\psi$ (\stacks{02TY} and \stacks{0BXJ}). Since $Z$ is cut out by a section of $\psi^*(\mc{O}_{\PP^k}(-n+2k-2))$, $[Z]$ generates the subgroup $(n-2k+2)\bb{Z} \subset \bb{Z}=\op{Cl}(Y)$, from which the desired conclusion follows.

To prove that $Z$ has the desired properties, we consider the natural map on $\PP^k$
\begin{equation}\label{eq:sym-o-E-to-E}
 \Sym^{k-1}U \oo \mc{E}^k_{n-2k+1} \lra \mc{E},
\end{equation}
defined analogously to (\ref{eq:EU-to-next-E}) by the multiplication $\Sym^{n-k+1}U \oo \Sym^{k-1}U \lra \Sym^n U$. Pulling back along $\psi$ and applying the natural map $\psi^*(\mc{E}) \lra \mc{O}_Y$, we obtain a morphism
\[
\xymatrix{
 \Sym^{k-1}U \oo \mc{O}_Y \ar@/^2pc/[rr]^-\Delta \ar[r] & \psi^*(\mc{E}) \oo \psi^*\left(\mc{E}^k_{n-2k+1}\right)^{\vee} \ar[r] & \psi^*\left(\mc{E}^k_{n-2k+1}\right)^{\vee} \\
 }
 \]
of vector bundles of rank $k$ on $Y$. Since $\det(\mc{E}^k_{n-2k+1}) = \mc{O}_{\PP^k}(n-2k+2)$, we have that $\det(\Delta)$ defines a section of $\psi^*(\mc{O}_{\PP^k}(-n+2k-2))$, so to conclude we need to check that $Z$ is the degeneracy locus of $\Delta$. The restriction of $\Delta$ to the fiber at a point $(f,[g])\in Y$ as in (\ref{eq:explicit-Y}) is given by a map
\[
  \Sym^{k-1}U \lra \ker\left(\D^{n-k+1}U \overset{\langle -,g\rangle}{\lra}\D^{n-2k+2}U\right),\quad h \mapsto \langle f,h \rangle.
\]
This map drops rank precisely when there exists a non-zero $h\in\Sym^{k-1}U$ with $\langle f,h \rangle = 0$, that is, when $(f,[h]) \in \bb{A}_{\PP^{k-1}}(\mc{E}^{k-1}_{n-k+1})$. By Lemma~\ref{lem:desingularization}, this happens if and only if $f\in\widehat{\Sigma}_{k-1}$, or equivalently $(f,[g])\in Z$.

Thus, $\det(\Delta)$ defines $Z$ set-theoretically. Hence $n-2k+2$ gives an upper bound on the size of the class group of $\widehat{\Sigma}_k$. But we will construct $n-2k+2$ different representatives in Lemma~\ref{lem:MCMs}, so in fact, $\det(\Delta)$ must define $Z$ scheme-theoretically as well.
\end{subequations}
\end{proof}

It is noted in \cite[Remark~3.7]{CMSV} that each class in $\op{Cl}(\widehat{\Sigma}_k)$ is represented by a rank one MCM (maximal Cohen--Macaulay) module, and an explicit construction is given in terms of ideals of minors of Hankel matrices. For an $\SL$-equivariant realization of these MCM modules we argue using the Kempf--Weyman geometric technique. We define
\begin{equation}\label{eq:def-Mr}
  M_r = \pi_*(\psi^*(\mc{O}_{\PP^k}(r))) = {\rm H}^0(Y, \psi^*(\mc{O}_{\PP^k}(r))) \text{ for }r=0,\dots,n-2k+1,
\end{equation}
and prove the following (note that $M_0=B$).

\begin{lemma}\label{lem:MCMs}
 Up to isomorphism, the rank one maximal Cohen--Macaulay $B$-modules are $M_0,\dots,M_{n-2k+1}$. The minimal number of generators of $M_r$ is
 \[ \mu(M_r) = {r + k \choose k},\]
 and $M_{n-2k+1}$ is an Ulrich module.
\end{lemma}

We recall that an MCM module $M$ is called \defi{Ulrich} (short for \defi{maximally generated maximal Cohen--Macaulay}) if $\mu(M)$ equals the multiplicity of $M$, see \cites{BHU,HUB} for basic information and references.

\begin{proof} Since $M_r$ is the direct image of a line bundle on the desingularization $Y$ of $\widehat{\Sigma}_k$, it follows that $M_r$ is a rank one $B$-module. To see that $M_r$ is Cohen--Macaulay, we apply \cite[Corollary~5.1.5]{weyman}. With the notation in loc. cit., if $\mc{V}=\mc{O}_{\PP^k}(r)$ then we have that $\mc{V}^{\vee} = \mc{O}_{\PP^k}(2n-k-r)$. We then have to check that
  \[
    {\rm R}^i\pi_*(\psi^*\mc{V}^{\vee}) = 0\text{ for }i>0.
  \]
Equivalently, we have to show that for each $d\geq 0$,
\[
  \rH^i(\PP^k,\Sym^d\mc{E} \oo \mc{O}_{\PP^k}(2n-k-r)) = 0\text{ for }i>0.
\]
By our assumption on $r$, we have $2n-k-r + 1 \geq 0$. The $2$-step resolution (\ref{eq:ses-Edm}) of $\mc{E}$ induces for each $d\geq 0$ a resolution of $\Sym^d(\mc{E})$, where the $i$-th step is isomorphic to a direct sum of line bundles of the form $\mc{O}_{\PP^k}(-i)$. This implies that $\Sym^d(\mc{E})$ is a $0$-regular sheaf. Therefore it is also $(2n-k-r + 1)$-regular, hence the desired vanishing statement holds.

To calculate $\mu(M_r)$, we consider the minimal free resolution $F^r_{\bullet}$ of $M_r$. As in the proof of Proposition~\ref{prop:syzygies-B}, we have using \cite[Theorem~5.1.2]{weyman} that
\[ 
\begin{aligned}
F^r_0 &= \bigoplus_{j\geq 0}\rH^j\left(\PP^k,\mc{O}_{\PP^k}(-j+r)\right) \oo \bw^{j}(\Sym^{n-k}U) \oo S(-j) \\
&=\rH^0\left(\PP^k,\mc{O}_{\PP^k}(r)\right) \oo S = \Sym^r(\D^k U) \oo S.
\end{aligned}
\]
The space of minimal generators of $M_r$ is then $\Sym^r(\D^k U)$, which has dimension ${r+k\choose k}$. 
Since the numbers $\mu(M_r)$ are distinct, it follows that the modules $M_r$ are pairwise non-isomorphic. Since $B$ is normal, the class group of $B$ is isomorphic to the group of isomorphism classes rank 1 reflexive $B$-modules \stacks{0EBM}. Since the class group of $B$ has size $(n-2k+2)$, it follows that every rank 1 MCM $B$-module is isomorphic to $M_r$ for some~$r$. 

Since $M_r$ is a rank one MCM module supported on $\widehat{\Sigma}_k$, it follows from Corollary~\ref{cor:dim-degree-Sk} that its multiplicity is ${n-k+1\choose k}$. This quantity agrees with $\mu(M_r)$ precisely when $r=n-2k+1$, proving that $M_{n-2k+1}$ is Ulrich.
\end{proof}

\section{Self-duality for the rank one Ulrich module}\label{sec:self-duality}

In this section we analyze the Ulrich module $M_{n-2k+1}$ in Lemma~\ref{lem:MCMs}, and prove that it is self-dual. To explain this, we note that the definition (\ref{eq:def-Mr}) can be extended to arbitrary $r\in\bb{Z}$, but the resulting modules will typically fail to be Cohen--Macaulay. With the notation in Lemma~\ref{lem:MCMs}, if $\mc{V}=\mc{O}_{\PP^k}(n-2k+1)$ then $\mc{V}^{\vee}=\mc{O}_{\PP^k}(-1)$, so in fact $M_{-1}$ is MCM and is the dual of $M_{n-2k+1}$. We will show that the section $\det(\Delta)$ constructed in the proof of Proposition~\ref{prop:class-group} gives rise to an isomorphism $M_{-1}\simeq M_{n-2k+1}$ which is given at the level of minimal generators by Hermite reciprocity.

\begin{theorem}\label{thm:self-dual-Ulrich}
 The graded modules $M_{-1}$ and $M_{n-2k+1}$ are generated in a single degree, with spaces of minimal generators
 \[ \bw^k(\Sym^{n-k}U) \text{ and }\Sym^{n-2k+1}(\D^k U)\]
 respectively. There is an isomorphism $M_{-1}\simeq M_{n-2k+1}$ which restricts to (\ref{eq:Hermite-skew-sym}) on the space of minimal generators.
\end{theorem}

\begin{proof}
 We begin by describing the minimal free resolutions $F^{n-2k+1}_{\bullet}$ and $F^{-1}_{\bullet}$ of $M_{n-2k+1}$ and $M_{-1}$ respectively. We obtain using \cite[Theorem~5.1.2]{weyman} that (up to a shift in grading)
 \[ 
 \begin{aligned}
 F^{n-2k+1}_i &= \rH^0\left(\PP^k,\bw^i(\Sym^{n-k}U(-1))\oo\mc{O}_{\PP^k}(n-2k+1)\right) \oo S(-i) \\
 & = \bw^i(\Sym^{n-k}U) \oo \Sym^{n-2k+1-i}(\D^k U) \oo S(-i),\text{ and }
 \end{aligned}
 \]
\[ 
 \begin{aligned}
 F^{-1}_i &= \rH^k\left(\PP^k,\bw^{i+k}(\Sym^{n-k}U(-1))\oo\mc{O}_{\PP^k}(-1)\right) \oo S(-i) \\
 & = \bw^{i+k}(\Sym^{n-k}U) \oo \D^i(\Sym^k U) \oo S(-i),\text{ for }i=0,\dots,n-k+1.
 \end{aligned}
 \]
By taking $i=0$ we obtain the desired description for the spaces of minimal generators of $M_{n-2k+1}$ and $M_{-1}$. Since $\mu(M_{n-2k+1})=\mu(M_{-1})$ and $M_{n-2k+1},M_{-1}$ are rank one MCMs, it follows from Lemma~\ref{lem:MCMs} that they must be isomorphic. Before describing the isomorphism, we make explicit the graded structure of the two modules, which is induced by the usual identification of $\mc{O}_Y$ with the sheaf of graded algebras $\Sym(\mc{E})$ on $\PP^k$. We have that
\begin{align*}
  (M_{n-2k+1})_d &= \rH^0(\PP^k,\Sym^d\mc{E} \oo \mc{O}_{\PP^k}(n-2k+1)),\text{ and }\\
  (M_{-1})_d &= \rH^0(\PP^k,\Sym^{d+k}\mc{E} \oo \mc{O}_{\PP^k}(-1)).
\end{align*}

To construct an isomorphism, we note that the section $\det(\Delta)$ of $\psi^*(\mc{O}_{\PP^k}(-n+2k-2))$ in Proposition~\ref{prop:class-group} defines an injective morphism of invertible sheaves on $Y$
\[
  \delta\colon \psi^*(\mc{O}_{\PP^k}(n-2k+1)) \lra \psi^*(\mc{O}_{\PP^k}(-1)).
\]
 Making the identification of the source and target with sheaves of graded $\Sym(\mc{E})$-modules, we see that in degree $d$ the map is given by
 \[
   \delta_d\colon \Sym^d\mc{E} \oo \mc{O}_{\PP^k}(n-2k+1) \lra \Sym^{d+k}\mc{E} \oo \mc{O}_{\PP^k}(-1).
 \]
Since $\delta_d$ is injective, the same is true for $\rH^0(\PP^k,\delta_d) \colon (M_{n-2k+1})_d \lra (M_{-1})_d$. Since the source and target are finite dimensional vector spaces, we see that $\delta$ induces an isomorphism between $M_{n-2k+1}$ and $M_{-1}$. To conclude, we need to check that $\delta_0$ induces (\ref{eq:Hermite-skew-sym}).

Note that $\delta_0\oo\mc{O}_{\PP^k}(1)$ is constructed from (\ref{eq:sym-o-E-to-E}), as the composition
\[ \mc{O}(n-2k+2) \simeq \bw^k(\mc{E}^k_{n-2k+1}) \simeq \bw^k(\Sym^{k-1}U) \oo \bw^k(\mc{E}^k_{n-2k+1}) \lra \Sym^k(\mc{E}).\]
The presentation (\ref{eq:ses-Edm}) induces a resolution $\mc{F}_{\bullet}$ of $\bw^k(\mc{E}^k_{n-2k+1})$, with
\[ \mc{F}_i = \D^i(\Sym^{n-2k+1}U) \oo \bw^{k-i}(\Sym^{n-k+1}U) \oo \mc{O}_{\PP^k}(-i),\]
and a resolution $\mc{G}_{\bullet}$ of $\Sym^k(\mc{E})$, with
\[ \mc{G}_i = \bw^i(\Sym^{n-k}U) \oo \Sym^{k-i}(\Sym^n U) \oo \mc{O}_{\PP^k}(-i).\]
We lift $\delta_0\oo\mc{O}_{\PP^k}(1)$ to a map of complexes $\mc{F}_{\bullet}\lra\mc{G}_{\bullet}$, using the comultiplication
\[ \kk \simeq \bw^k(\Sym^{k-1}U) \lra \bw^i(\Sym^{k-1}U) \oo \bw^{k-i}(\Sym^{k-1}U),\]
and the natural maps
\[ \D^i(\Sym^{n-2k+1}U) \oo \bw^i(\Sym^{k-1}U) \lra \D^i(\Sym^{n-2k+1}U\oo\Sym^{k-1}U) \lra \D^i(\Sym^{n-k}U)\]
and
\[ \bw^{k-i}(\Sym^{n-k+1}U) \oo \bw^{k-i}(\Sym^{k-1}U) \lra \Sym^{k-i}(\Sym^{n-k+1}U \oo \Sym^{k-1}U) \lra \Sym^{k-i}(\Sym^n U).\]
We get a morphism of exact complexes
\[ 
\xymatrix{
0 \ar[r] & \mc{F}_k(-1) \ar[r] \ar[d] & \cdots \ar[r] & \mc{F}_0(-1) \ar[r] \ar[d] & \mc{O}_{\PP^k}(n-2k+1) \ar[r] \ar[d]^{\delta_0} & 0\\
0 \ar[r] & \mc{G}_k(-1) \ar[r] & \cdots \ar[r] & \mc{G}_0(-1) \ar[r] & \Sym^k(\mc{E})\oo \mc{O}_{\PP^k}(-1) \ar[r] & 0\\
}
\]
and note that the sheaves $\mc{F}_i(-1)$ and $\mc{G}_i(-1)$ have no cohomology for $0\leq i\leq k-1$, while $\mc{F}_k(-1)$, $\mc{G}_k(-1)$ only have top cohomology. Taking hypercohomology we obtain a commutative square
\[
\xymatrix{
\rH^k(\PP^k,\mc{F}_k(-1)) = \D^k(\Sym^{n-2k+1}U) \ar[rr]^-{\gamma} \ar[d]_-{\alpha} & & (M_{n-2k+1})_0= \Sym^{n-2k+1}(\D^k U) \ar[d]^-{\rH^0(\PP^k,\delta_0)}  \\
\rH^k(\PP^k,\mc{G}_k(-1)) = \bw^k(\Sym^{n-k}U) \ar@{=}[rr] & & (M_{-1})_0\\
}
\]
The vertical map $\a$ is by construction the isomorphism in Section~\ref{subsec:comp-Her}, while $\gamma$ is the Hermite isomorphism (\ref{eq:Hermite-sym-div}), as constructed in Section~\ref{sec:Hermite-Hilbert}. We wrote the bottom map as an equality, because this is how we identify $(M_{-1})_0$ with $\bw^k(\Sym^{n-k}U)$. It follows from Theorem~\ref{thm:comp-Her} that
\[
  \rH^0(\PP^k,\delta_0) = \alpha\circ\gamma^{-1} = \beta^{-1},
\]
is the inverse of the Hermite isomorphism (\ref{eq:Hermite-skew-sym}), concluding our proof.
\end{proof}

In \cite[Proposition~1.2]{valles}, it is shown that $\Sym^2(\mc{E}^2_{n-2})\oo \mc{O}_{\PP^2}(-n+2)$ has non-zero global sections. We show that in fact it has only one non-zero section up to scaling, and that this property extends to higher rank Schwarzenberger bundles as follows.

\begin{proposition}\label{prop:h0-SkE}
 For $n\geq 2k-1$ we have that $\rH^0(\PP^k,\Sym^k(\mc{E}) \oo \mc{O}_{\PP^k}(-n+2k-2))$ is one-dimensional.
\end{proposition}

\begin{proof}
  The proof of Theorem~\ref{thm:self-dual-Ulrich} shows that $\rH^0(\PP^k,\Sym^k(\mc{E}) \oo \mc{O}_{\PP^k}(-n+2k-2))$ contains at least one non-zero section, and that any such section yields an isomorphism $M_{n-2k+1}\simeq M_{-1}$. Suppose now that $s_1,s_2$ are non-zero sections, write $M=M_{n-2k+1}$, and let $s=s_1^{-1}\circ s_2$ denote the induced automorphism of $M$. Since $M$ is a rank one reflexive module, we have an isomorphism $\Hom_B(M,M)\simeq B$. Moreover, identifying $M$ with a fractional ideal, we can represent $s$ as the multiplication by an element in the fraction field $\op{Frac}(B)$, which we also denote by $s$. Since $sM\subseteq M$ and $M$ is a finitely generated $B$-module, it follows from the Cayley--Hamilton theorem that $s$ is integral over $B$. By Theorem~\ref{thm:Hankel}, $B$ is normal, hence $s\in B$. Since $M$ is graded, and multiplication by $s$ is an isomorphism on $M$, it follows that $s$ must have degree $0$, that is, $s\in\kk$. It follows that $s_1,s_2$ are proportional, concluding our proof.
\end{proof}

The following remark explains how to think of the unique, up to scaling, non-zero section of $\Sym^k(\mc{E}) \oo \mc{O}_{\PP^k}(-n+2k-2)$ as Hermite reciprocity.
 
\begin{remark}
 Using the resolution $\mc{G}_{\bullet}$ of $\Sym^k(\mc{E})$ in the proof of Theorem~\ref{thm:self-dual-Ulrich}, we can identify $\rH^0(\PP^k,\Sym^k(\mc{E}) \oo \mc{O}_{\PP^k}(-n+2k-2))$ as the kernel of the natural map
 \[
   \rH^k(\PP^k,\mc{F}_k(-n+2k-2)) \lra \rH^k(\PP^k,\mc{F}_{k-1}(-n+2k-2)).
 \]
 We have
 \[ 
 \begin{aligned}
 \rH^k(\PP^k,\mc{F}_k(-n+2k-2)) &= \bw^k(\Sym^{n-k}U) \oo \D^{n-2k+1}(\Sym^k U) \\
 &= \Hom_\kk\left(\Sym^{n-2k+1}(\D^k U),\bw^k(\Sym^{n-k}U)\right).
 \end{aligned}
\]
Our results can be summarized as saying that $\rH^0(\PP^k,\Sym^k(\mc{E}) \oo \mc{O}_{\PP^k}(-n+2k-2))$ is naturally identified with the subspace of $\Hom_\kk\left(\Sym^{n-2k+1}(\D^k U),\bw^k(\Sym^{n-k}U)\right)$ spanned by the inverse of the Hermite isomorphism (\ref{eq:Hermite-skew-sym}).
\end{remark} 
 
The isomorphism between $M_{n-2k+1}$ and $M_{-1}$ lifts to an isomorphism between their free resolutions, leading up to a series of $\SL$-isomorphisms that extend Theorem~\ref{thm:Hermite}.

\begin{corollary}\label{cor:gen-Her-isos}
 For each $i=0,\dots,n-k+1$ we have $\SL$-equivariant isomorphisms
 \[
   \bw^i(\Sym^{n-k}U) \oo \Sym^{n-2k+1-i}(\D^k U)  \simeq \bw^{i+k}(\Sym^{n-k}U) \oo \D^i(\Sym^k U).
 \]
\end{corollary} 

\begin{proof} The two sides of the isomorphism above correspond to the minimal generators of the free modules $F^{n-2k+1}_i$ and $F^{-1}_i$ in the resolutions of $M_{n-2k+1}$ and $M_{-1}$ constructed in the proof of Theorem~\ref{thm:self-dual-Ulrich}.
\end{proof}

\section{Syzygies of canonical curves}\label{sec:syzygies}

In this section we summarize some recent applications of Hermite reciprocity to Green's conjecture for generic canonical curves (see \cite{green} for the original reference, and \cites{eis-orientation,ein-laz} and \cite[Chapter~9]{DE-syzygies} for some expository accounts). The first proof is due to Voisin \cite{voisin-JEMS, voisin-odd} (using different methods), and a streamlined version of Voisin's arguments was recently explained in \cite{kemeny}. The approaches we discuss here are more elementary and also deal with the problem in positive characteristics. Throughout the section we work over an algebraically closed field $\kk$ and assume that $\rm{char}(\kk) \ne 2$.

Let $C$ be a smooth curve of genus $g\geq 3$ over $\kk$, and let $\omega_C$ be its canonical bundle. The canonical ring $S_C = \bigoplus_{n \ge 0}\rH^0(C, \omega_C^{\otimes n})$ is finitely generated over the $g$-dimensional polynomial ring $S = \Sym(\rH^0(C, \omega_C))$ and hence we can define the graded Betti numbers
\[
  \beta_{i,j}(C,\omega_C) = \dim_\kk \Tor_i^S(\kk, S_C)_j.
\]
Green's conjecture asserts that, when $\rm{char}(\kk)=0$, we have $\beta_{i,i+2}(C,\omega_C) = 0$ for $i<\rm{Cliff}(C)$, where $\rm{Cliff}(C)$ is the Clifford index of $C$. We do not need the definition of the Clifford index here, but instead note that for most curves, we have $\rm{Cliff}(C) = d - 2$ where $d$ is the gonality of~$C$ (recall that the \defi{gonality} of an algebraic curve $C$ is the minimum degree of a non-constant map from $C$ to $\PP^1$; here ``most'' means that it holds for a Zariski open subset of the locus of curves of each fixed gonality in the moduli space of curves).

To show that Green's conjecture holds generically, that is, for a non-empty Zariski open subset of curves in the moduli space of curves, one can appeal to degeneration techniques and check that it is satisfied by a single (smoothable) curve. Several examples of such curves have been proposed over the years, including rational cuspidal curves suggested independently by Buchweitz--Schreyer and O'Grady, and ribbon curves put forward by Bayer and Eisenbud~\cite{BE}. We discuss these examples separately in Section~\ref{subsec:cusp} and~\ref{subsec:ribbon} below. Examples of Schreyer show that Green's conjecture does not hold in positive characteristic~\cite{schreyer-green}, but a careful analysis of cuspidal and ribbon curves allows one to keep track effectively of the characteristics where Green's conjecture holds generically. 

\subsection{Rational cuspidal curves}\label{subsec:cusp}

This section follows \cite{AFPRW}. A rational curve with $g$ simple cusps has genus $g$ and can be smoothed out, i.e., there exist flat families whose special fiber is a rational cuspidal curve $C$ and whose generic fiber is smooth. The upper bound for the Clifford index of a genus $g$ curve is $\lfloor (g-1)/2 \rfloor$. This means that it suffices to show that $\beta_{i,i+2}(C,\omega_C) = 0$ for $i < \lfloor (g-1)/2 \rfloor$. We can realize $C$ as a hyperplane section of the tangential variety of the rational normal curve in its $g$-uple embedding, so the computation can equivalently be done for the tangential variety. The advantage of the latter is that it has $\SL_2$-symmetry. 

To be precise, the \defi{tangential variety} $T_g$ is the union of all tangent lines to the rational normal curve $\Gamma$ in $\PP^g = \mathbb{P}(\Sym^g U)$. We have an inclusion of sheaves over $\PP^1 = \mathbb{P}(U)$ given by $(\Sym^{g-2} U)(-2) \to \Sym^g U \otimes \mc{O}_{\PP^1}$ whose cokernel $\mc{J}$ is the \defi{bundle of principal parts} of $\mc{O}_{\PP^1}(g)$ \cite[Section~7.2]{3264}. This induces a map $\PP(\mc{J}) \to \PP^g$ which is birational onto $T_g$.

It turns out that there is a short exact sequence
\[
  0 \to \kk[T_g] \to \widetilde{\kk[T_g]} \to \omega_\Gamma(-1) \to 0,
\]
where $\kk[T_g]$ is the homogeneous coordinate ring of $T_g$, $\widetilde{\kk[T_g]}$ is its normalization, and $\omega_\Gamma$ is the canonical module of the homogeneous coordinate ring of $\Gamma$. Hence we can use \cite[Theorem 5.1.2]{weyman} with $V = \PP^1$, $X = \bb{A}(\Sym^g U)$, $\xi = (\Sym^{g-2} U)(-2)$, and $\mc{V}=\mc{O}_{\PP^1}$ to compute Tor groups for $\widetilde{\kk[T_g]}$ over $S = \Sym(\Sym^g U)$:
\begin{align*}
  \Tor_i^S(\kk,\widetilde{\kk[T_g]})_{i+1} &= \D^{2i} U \oo \bigwedge^{i+1}(\Sym^{g-2} U) \qquad \text{for } i=0,\dots,g-2.
\end{align*}
All other Tor groups vanish except $\Tor_0^S(\kk,\widetilde{\kk[T_g]})_0 = \kk$.
The canonical module $\omega_\Gamma$ can be realized as $\bigoplus_{n \ge 0} \rH^0(\PP^1, \mc{O}_{\PP^1}(ng+g-2))$, so again using \cite[Theorem 5.1.2]{weyman} with $V = \PP^1$, $X = \bb{A}(\Sym^g U)$, $\xi = (\Sym^{g-1} U)(-1)$, and $\mc{V}= \mc{O}_{\PP^1}(g-2)$, its Tor groups can be computed:
\[
  \Tor_i^S(\kk,\omega_\Gamma(-1))_{i+1} = \Tor_i^S(\kk,\omega_\Gamma)_i =  \bigwedge^{i}(\Sym^{g-1} U) \otimes \Sym^{g-2-i}(U) \qquad \text{for } i=0,\dots,g-2.
\]

Using the long exact sequence on Tor, our goal then is to show that the map
\begin{equation}\label{eq:delta-tor}
  \Tor_i^S(\kk,\widetilde{\kk[T_g]})_{i+1} \to \Tor_i^S(\kk,\omega_\Gamma(-1))_{i+1}
\end{equation}
is surjective for $i \le \lfloor (g-1)/2 \rfloor$. Using our computations above, the terms are
\[
  \D^{2i} U \oo \bigwedge^{i+1}(\Sym^{g-2} U) \to \bigwedge^i(\Sym^{g-1} U) \otimes \Sym^{g-2-i}(U).
\]
Describing these maps directly is quite subtle! What is done in \cite{AFPRW} is to realize
\[\bigwedge^i(\Sym^{g-1} U) \otimes \Sym^{g-2-i}(U) = \ker\left(\D^{i+1}U\oo\bw^{i+1}(\Sym^{g-1}U) \overset{\eqref{eqn:hermite-mult}}{\lra} \bw^{i+1}(\Sym^{g}U)\right),\]
and rewrite \eqref{eq:delta-tor} as the (middle) homology of a $3$-term complex
\[ \D^{2i} U \oo \bigwedge^{i+1}(\Sym^{g-2} U) \to \D^{i+1}U\oo\bw^{i+1}(\Sym^{g-1}U) \overset{\eqref{eqn:hermite-mult}}{\lra} \bw^{i+1}(\Sym^{g}U),\]
where the first map comes from the inclusion $\D^{2i} U \to \D^{i+1}U \oo \D^{i+1}U$, followed by \eqref{eqn:hermite-mult}. Taking the direct sum over all $g$ as in Section~\ref{sec:hermite-alg}, and using the Hermite isomorphism, this gets identified with a $3$-term complex of free modules over $\tilde{S} = \Sym(\D^{i+1} U)$:
\[ \D^{2i}U \oo \tilde{S}(-2) \lra \D^{i+1}U \oo \tilde{S}(-1) \lra \tilde{S},\]
which a subcomplex of the Koszul complex for the maximal ideal of $\tilde{S}$. The (middle) homology is a \defi{Weyman module $W^{(i+1)}$}, which is a special case of a \defi{Koszul module} considered in \cite{AFPRW}. Using general vanishing results for Koszul modules, for which we refer the reader to \cite[Section~2]{AFPRW}, one gets that \eqref{eq:delta-tor} is surjective when $i \le \lfloor (g-1)/2 \rfloor$ (and $\rm{char}(\kk)=0$ or $\rm{char}(\kk) \ge \frac{g+2}{2}$), as desired.



\subsection{Ribbon curves}\label{subsec:ribbon}

This section follows \cite{SR-bigraded}. A ribbon structure on a variety $X$ is a non-reduced scheme whose structure sheaf is a square-zero extension of $\mc{O}_X$ by some line bundle on $X$ \cite[Section~1]{BE}. We will be interested in ribbon structures on $\PP^1$, which are called \defi{rational ribbon curves}. They offer more flexibility than cuspidal curves, in that they allow us to keep track not only of the genus, but also of the gonality.

Given an integer $a \le (g-1)/2$, we consider the projective space 
\[\PP^{g}=\mathbb{P}(\Sym^a U \oplus \Sym^{g-1-a} U).\]
We have the Veronese embeddings $v_1 \colon \PP^1 \to \mathbb{P}(\Sym^a U)$ and $v_2 \colon \PP^1 \to \mathbb{P}(\Sym^{g-1-a} U)$ into the corresponding subspaces, and for each $x\in \PP^1$ we get a line joining $v_1(x)$ and $v_2(x)$. The \defi{rational normal scroll} $\mc{S}(a,g-1-a)$ is the union of these lines as we vary over $x \in \PP^1$ \cite{EH-centennial}. Let $B$ be the coordinate ring of $\mc{S}(a,g-1-a)$. There is a ribbon structure $\mc{X}(a,g-1-a)$ on $\mc{S}(a,g-1-a)$, called a \defi{K3 carpet}, whose coordinate ring $A$ fits into a short exact sequence
\[
  0 \to \omega_B \to A \to B \to 0
\]
where $\omega_B$ is the canonical module of $B$. Most importantly, a hyperplane section of the K3 carpet $\mc{X}(a,g-1-a)$ gives a genus $g$ canonically embedded ribbon that can be smoothed out to a curve of Clifford index $a$. As in Section~\ref{subsec:cusp}, we can reduce to the study of the syzygies of the K3 carpet, which like $T_g$, has $\SL_2$-symmetry.

The multiplication map
\[
  U \otimes (\Sym^{a-1} U \oplus \Sym^{g-2-a} U) \to \Sym^a U \oplus \Sym^{g-1-a} U
\]
gives a $2 \times (g-1)$ matrix whose entries are linear forms in $\PP^g$, and the Eagon--Northcott complex for this matrix gives a minimal free resolution for $B$. Hence, if we let
\[S = \Sym(\Sym^a U \oplus \Sym^{g-1-a} U)\]
denote the coordinate ring of $\PP^g$, then we have for $i \ge 1$
\begin{align*}
  \Tor_i^S(\kk, B)_{i+1} &= \D^{i-1} U \otimes \bigwedge^{i+1}(\Sym^{a-1} U \oplus \Sym^{g-2-a} U),\\
  \Tor_i^S(\kk,\omega_B)_{i+2} &= \Sym^{g-3-i} U \otimes \bigwedge^i(\Sym^{a-1} U \oplus \Sym^{g-2-a} U).
\end{align*}
As in the previous section, using the long exact sequence on Tor, our goal becomes to show that the connecting homomorphism
\begin{equation}\label{eq:delta-Tor-for-B}
  \Tor_{i+1}^S(\kk,B)_{i+2} \to \Tor_i^S(\kk,\omega_B)_{i+2}
\end{equation}
is surjective for $i < a$. However, it is better to view all the Tor groups as being bigraded, using the usual decomposition of exterior powers 
\[\bigwedge^n(V \oplus W) = \bigoplus_{n'+n''=n} \bigwedge^{n'} V \otimes \bigwedge^{n''} W,\]
with $V=\Sym^{a-1} U$ and $W=\Sym^{g-2-a} U$. The map then looks like
\begin{align}
  \bigoplus_{\substack{u+v=i\\u,v \ge -1}} \D^i U \otimes \bigwedge^{u+1}(\Sym^{a-1} U) \otimes \bigwedge^{v+1} (\Sym^{g-2-a} U)\to \label{eq:source-TorB} \\ 
  \bigoplus_{u+v=i} \Sym^{g-3-i} U \otimes \bigwedge^u(\Sym^{a-1} U) \otimes \bigwedge^v( \Sym^{g-2-a} U), \label{eq:target-Tor-omegaB}
\end{align}
and we can concentrate on a specific $(u,v)$-bigraded component, while taking the direct sum over all $a,g$. Via Hermite reciprocity as in Section~\ref{sec:hermite-alg}, the source \eqref{eq:source-TorB} becomes a free module over $\tilde{S}=\Sym(\D^{u+1} U \oplus \D^{v+1} U)$, more precisely $\D^{u+v}U \oo \tilde{S}(-1,-1)$. The target \eqref{eq:target-Tor-omegaB} also becomes a finitely generated $\tilde{S}$-module, though this is not at all obvious! In fact, it can be identified with the (middle) homology of 
\[ \D^{u+v+2}U\oo\tilde{S}(-1,-1) \lra \D^{u+1}U\oo\tilde{S}(-1,0) \oplus \D^{v+1}U\oo\tilde{S}(0,-1) \lra \tilde{S},\]
where the maps are now completely transparent ($\D^{u+v+2}U$ embeds into $\D^{u+1}U\oo\D^{v+1}U$ and $\D^{v+1}U\oo\D^{u+1}U$ via comultiplication). Letting $Q_{u,v}=\D^{u+v+2}U\oplus\D^{u+v}U$ (or rather an appropriate $\SL_2$-equivariant extension of the summands) leads to a subcomplex of the (bi-graded) Koszul complex of $\tilde{S}$: 
\[ Q_{u,v}\oo\tilde{S}(-1,-1) \lra \D^{u+1}U\oo\tilde{S}(-1,0) \oplus \D^{v+1}U\oo\tilde{S}(0,-1) \lra \tilde{S}.\]
Its middle homology is the \defi{bi-graded Weyman module $W^{(u+1,v+1)}$}, a special instance of a \defi{bi-graded Koszul module}, for which appropriate vanishing theorems are established in \cite[Section~3]{SR-bigraded}. Based on this, it can be shown that the desired surjectivity of \eqref{eq:delta-Tor-for-B} holds when $\rm{char}(\kk)=0$ or $\rm{char}(\kk) \ge a$, but we refer to \cite{SR-bigraded} for details.


We finish by noting that this improves the approach via cuspidal curves in two ways. First, the bound on the characteristic for generic curves is slightly better than the one given by rational cuspidal curves since there we have $a = \lfloor (g-1)/2 \rfloor$, and in particular it confirms a conjecture from \cite{ES-carpets}. Second, this allows us to prove that generic Green's conjecture holds for each gonality, and not just the maximum value.

\section*{Acknowledgements}
We thank Giorgio Ottaviani, Rob Lazarsfeld and Jerzy Weyman for helpful discussions. Experiments with the computer algebra software Macaulay2 \cite{GS} have provided numerous valuable insights.

\end{document}